\documentclass{amsart}

\usepackage{verbatim,amsmath,amssymb}

\theoremstyle{plain}   %% This is the default, anyway
\begingroup % Confine the \theorembodyfont command
\newtheorem{theorem}{Theorem}[section]   % Numbered within each section

\newtheorem{corollary}[theorem]{Corollary}     % Numbered along with thm
\newtheorem{lemma}[theorem]{Lemma}         % Numbered along with thm
\newtheorem{proposition}[theorem]{Proposition}  % Numbered along with thm

\newtheorem{definition}[theorem]{Definition}
\newtheorem*{question}{Question}

\endgroup

\theoremstyle{definition}

\theoremstyle{remark}
\newtheorem{remark}[theorem]{Remark}        % Numbered along with thm
\newtheorem{example}[theorem]{Example}        % Numbered along with thm

\newcommand{\is}{\operatorname{isol}}
\newcommand{\length}{\operatorname{length}}

\newcommand{\vf}{\varphi}
\newcommand{\spa}{\operatorname{span}}
\newcommand{\cal}{\mathcal}

\newcommand{\Dom}{\operatorname{Dom}}

\newcommand{\dist}{{\rm dist}\,}

\newcommand{\R}{{\mathbb R}}
\newcommand{\N}{{\mathbb N}}
\newcommand{\INt}{{\rm int}\,}
\newcommand{\ep}{\varepsilon}

\newcommand{\diam}{{\rm diam}\,}
\newcommand{\Nor}{{\rm Nor}\,}
\newcommand{\Tan}{{\rm Tan}\,}

\newcommand{\reach}{{\rm reach}\,}

\newcommand{\Unp}{\operatorname{Unp}}

\newcommand{\graph}{\operatorname{graph}}
\newcommand{\card}{\operatorname{card}}

\newcommand{\conv}{\operatorname{conv}}
\newcommand{\intr}{\operatorname{Int}}

\newcommand{\hypo}{\operatorname{hypo}}
\newcommand{\epi}{\operatorname{epi}}

\def\halfsq{\hbox{\kern1pt\vrule height 7pt\vrule width6pt height 0.4pt depth0pt\kern1pt}}
\def\ihalfsq{\hbox{\kern1pt \vrule width6pt height 0.4pt depth0pt
                   \vrule height 7pt \kern1pt}}
\begin{document}

\title{On the structure of sets with positive reach}
\author{Jan Rataj, Lud\v ek Zaj\'\i \v cek}
\address{Charles University\\
Faculty of Mathematics and Physics\\
Sokolovsk\'a 83\\
186 75 Praha 8\\
Czech Republic}

\email{rataj@karlin.mff.cuni.cz}
\email{zajicek@karlin.mff.cuni.cz}

\thanks{The authors were supported by the grant GA\v CR No.~P201/15-08218S}

\begin{abstract}
We give a complete characterization of compact sets with positive reach (=proximally $C^1$ sets) in the plane and of one-dimensional sets with positive reach in $\R^d$. 
 Further, we prove that 
 if  $\emptyset \neq A\subset \R^d$  is a  set of positive reach of topological dimension
 $0< k < d$, then $A$ has its ``$k$-dimensional regular part''  $\emptyset \neq R \subset A$ which  is a $k$-dimensional ``uniform'' $C^{1,1}$ manifold open in $A$  and
 $A\setminus R$  can be locally covered by finitely many  $(k-1)$-dimensional DC surfaces.
 We  also show that if $A \subset \R^d$ has positive reach, then
 $\partial A$ can be locally covered by finitely many semiconcave hypersurfaces.
\end{abstract}

%\keywords{...}
%\subjclass[2000]{...}
%\date{\today}
\maketitle

\section{Introduction}
Federer in his fundamental paper \cite{Fe59} unified the approaches of convex and differential geometry, introducing curvature measures for sets with positive reach and proving the kinematic formulas. Sets with positive reach were also studied under distinct names (e.g., ``proximally smooth
 sets'' or ``prox-regular sets'') in general Hilbert spaces, cf. \cite{CT10}.

Of course, since  sets with positive reach form an important class, there exists
 a number of interesting results on their structure.

First we mention several results on the structure of ``special'' sets with positive reach. As far as we know, the first interesting result in this direction is essentially contained
 in Rechetnyak's 1956 paper \cite{Re} published  before Federer's seminal work. This result which was proved independently, using modern terminology, in \cite{Fu85}, reads as follows:
\medskip

{\bf (A)}\ \ \ \ \ {\it If $f: \R^{d-1} \to \R$ is Lipschitz, then its (closed) subgraph
 has positive reach if and only if $f$ is semiconcave.}
\medskip

The following related result was stated (in other words) without a proof in \cite[Remark 4.20]{Fe59}:
\medskip
 
{\bf (B)}\ \ \ \ {\it If $A\subset \R^d$ is a Lipschitz manifold of dimension $0<k<d$ with positive reach, then $A$ is a $k$-dimensional $C^{1,1}$ manifold.}
\medskip

The claims of \cite[Remark 4.20]{Fe59} easily imply also a more general result.
\medskip

{\bf (C)}\ \ \ \ {\it If $A\subset \R^d$ is a topological manifold of dimension $0<k<d$ with positive reach, then $A$ is a $k$-dimensional $C^{1,1}$ manifold.}

\medskip  

A proof of (C) was given by A. Lytchak, see \cite[Proposition 1.4]{Ly05} (even in Riemannian manifolds). It is based on a 
 Federer's unproved claim (which is a consequence of (E) below) and on the theory of length spaces (namely
$CAT(\kappa)$ spaces). Proofs of (B) for $k=d-1$ are well-known (cf. \cite[p. 3]{Scho} or \cite{Da});
 we prove also the general case, see Remark \ref{plb} below. 

The following result on special sets with positive reach was proved in \cite{CF}:
\medskip

{\bf (D)}\ \ \ \ {\it If  $\emptyset \neq A\subset \R^2$  is a connected set with positive reach and empty interior, then $A$ is either a singleton or a 1-dimensional manifold
 (possibly with boundary) of class $C^{1,1}$.} 

\medskip

It is written in  \cite{CF} that the result (D) ``gives a  complete characterization
 of connected sets of  positive reach with empty interior in the plane'', however it
 is not true for unbounded sets, see Example \ref{smycka} below. In Section \ref{oned}  we generalize (D) giving
 a complete characterization of one-dimensional sets $A \subset \R^d$ with positive reach.

For general sets with positive reach, there exist several complete characterizations, e.g. Federer's
 characterization (Proposition \ref{fedtan} below) or Lytchak's characterization for compact sets
 (\cite[Theorem 1.3]{Ly05}). However, the structure of sets with positive reach can be rather complicated, and the above characterizations do not give a satisfactory answer to the question
 ``how their structure can be complicated''. Our main aim is to give some partial answers to this  (unprecise) question.

In $\R^2$, we give (see Section \ref{plane}) an almost satisfactory answer: we provide a simple complete characterization
 of the local structure of compact sets with positive reach. This is our first main result.

Our second main result on the structure of general sets $A \subset \R^d$ with positive reach
 is an improvement of a further result claimed by Federer in \cite[Remark 4.20]{Fe59}. This
 result works with sets
$$ A^{(k)}:= \{a\in A: \dim (\Nor (A,a)) \geq d-k\},$$
where $0\leq k \leq d$ and $\Nor (A,a)$ denotes the normal cone of $A$ at $a$:

\medskip                  

{\bf (E)}\ \ \ \ {\it If  $A\subset \R^d$  is a  set of positive reach of topological dimension
 $0< k \leq d$, then the set $A\setminus A^{(k-1)}$ 
 is open if $k=d$ and it is a $k$-dimensional $C^{1,1}$ manifold open in $A$ for $k<d$.}
\medskip

We will call $R:= A\setminus A^{(k-1)}$  ``the main regular part of $A$''. Federer proved
 that $A^{(k-1)}$ is  countably $(k-1)$-rectifiable; so
\begin{equation}\label{rozkl}
  A= R \cup A^{(k-1)}
\end{equation}
 is a (canonical) decomposition of $A$ to a regular (smooth) $k$-dimensional part $R$
 and  a remaining $(k-1)$-dimensional part $A^{(k-1)}$. 

We slightly improve (E) showing that $R$ is even a ``uniform $k$-dimensional $C^{1,1}$ manifold''
 (see Definition \ref{plochy} (e) and Theorem \ref{hlav}).

Further we show (Theorem \ref{hlav}) that
\medskip

{\bf (F)}\ \ \ \ {\it $A^{(k-1)}= A \setminus R$ can be locally covered by finitely many
  $(k-1)$-dimensional DC surfaces.}
	\medskip
	
	In particular, the set $A^{(k-1)}$ is not only $(k-1)$-dimensional, but it has even locally
	 finite $(k-1)$-dimensional Hausdorff measure.

	We obtain (F) as a consequence of the fact  that some singular sets of convex  functions can be covered by {\it finitely many} DC surfaces. The proofs of these results
 which  refine the arguments of \cite{zajcon} are contained in Section \ref{spcf}.
In fact, these results on singular sets were originally motivated and obtained during our research in progress with D. Pokorn\'y on WDC sets, which provide a natural generalization of sets with positive reach, see \cite{PR13}.
 By the same method 
we prove that a $k$-dimensional set of positive reach can be locally covered by finitely many  $k$-dimensional DC surfaces.

In case $d=k$ we  improve (F) showing that  the boundary
 $\partial A= A^{(d-1)} $ of a set of positive reach in $\R^d$ can be locally covered by finitely many semiconvex
 hypersurfaces. We prove that result (Theorem \ref{pokrhr}) directly, without using results on singularities of convex functions.

Using (E), we also observe (Corollary  \ref{huhl}) that each $A \subset \R^d$ with positive reach 
 has a ``smooth part'' (of non-constant dimension) which is open and dense in $A$.

\section{Preliminaries}
\subsection{Basic definitions}
The symbols $B(c,r)$ and $\overline B(c,r)$ denote open and closed ball of center $c$ and radius $r$,
 respectively. We also sometimes use notation $B^X(c,r)$ for the ball in the space $X$.
The closure of a set $A$ is denoted by $\overline A$ or $\operatorname{cl}(A)$ and the set of isolated points of $A$ by  $\operatorname{isol} A$.
 The symbol $[x,y]$ denotes the (closed) segment if $x,y\in\R^d$.
The symbol $\Pi_M$ stands for the metric projection, see \eqref{mproj}.
We consider only real Banach spaces; the norm is always denoted by $|\cdot|$. By $\spa\, M$ we denote the linear span of the set $M$ and
 by $S_X$  the unit sphere in $X$. If  $X= W \oplus V$, then $\pi_W$ is the projection on $W$
 along $V$. If $X$ is a Hilbert space and $V$ is not specified, we mean that $V= W^{\perp}$.
 If $x\in X$ and $x\in X^*$, we set  $\langle x, x^* \rangle := x^*(x)$. The scalar product
 of vectors $x$, $y$ is also denoted by $\langle x, y \rangle$. If $X$ is a Hilbert space, we
 identify by the standard way $X$ and $X^*$. The symbol $\mathcal H^k $ stands for the $k$-dimensional
 Hausdorff measure. For sets $A\subset \R^d$, we denote by $\dim (A)$ and $\dim_H (A)$
 the topological and Hausdorff dimensions, respectively. We use the notation $e_i$ for the $i$th canonical basis vector in $\R^d$, $i=1\ldots,d$.

A mapping is called $K$-Lipschitz if it is Lipschitz with a (not necessarily minimal) constant $K$. 
% A bijection $f$ is called bilipschitz ($K$-bilipschitz) if both $f$ and $f^{-1}$ are Lipschitz %($K$-Lipschitz).

If $f$ is a real function, we use the abbreviated notation $\{f\leq r\}$ for $\{x\in\Dom(f):\, f(x)\leq c\}$. The hypograph and epigraph of $f$ are defined as
\begin{eqnarray*}
\hypo f&:=&\{(x,t)\in\Dom(f)\times\R:\, f(x)\geq t\},\\
\epi f&:=&\{(x,t)\in\Dom(f)\times\R:\, f(x)\leq t\}.
\end{eqnarray*}
If $f$ is defined on an open subset of a normed linear space $X$, we use the notation
$f'_+(x,v)$ for the one sided directional derivative of $f$ at $x$ in direction $v$. If $f$ is, in addition, locally Lipschitz, the generalized directional derivative of $f$ at $x\in\Dom(f)$ is defined as
$$f^\circ(x,v):=\limsup_{y\to x,t\to 0_+}\frac{f(y+t)-f(y)}{t},\quad v\in X,$$
and the {\it Clarke's subgradient of $f$ at $x$} is
$$\partial f(x):=\{ u^*\in X^*:\, \langle u^*,v\rangle\leq f^\circ(x,v)\text{ for all }v\in X\}$$ 
(cf.\ \cite[Section~1.2]{Cl90}). 

A mapping $F$ between Banach spaces $X$, $Y$ is called to be $C^{1,1}$, if
 it is Fr\' echet differentiable and
 the derivative $ F': \Dom(F) \to  \mathcal L(X,Y)$ is Lipschitz.

\subsection{Semiconvex functions, DC functions (mappings) and corresponding surfaces}

 One of several natural equivalent definitions  (cf. \cite[Definition 1.1.1 and Proposition 1.1.3]{CS}) of semiconcavity  reads as follows. We formulate it in the generality we need.
 
 \begin{definition}\label{semieuk} \rm
 A real function $u$ on an open convex subset $C$ of a finite-dimen\-sio\-nal Hilbert space $X$  is called {\it semiconcave}  with a {\it  semiconcavity constant} $c\geq 0$ if $u$ is continuous on $C$ and
  the function $g(x) = u(x) - (c/2) |x|^2$ is  concave on $C$.
	
A real function 	$v$ on $C$ is called {\it semiconvex} (with a {\it  semiconvexity constant} $c\geq 0$) if $-v$ is semiconcave  (with a  semiconcavity constant $c$).
\end{definition}

If $u$, $C$ and $X$ are as 
in the above definition, then (see, e.g., \cite[Proposition 2.1.2 and Corollary 3.3.8]{CS})
\begin{equation}\label{sscjj}
\text{$u$ is $C^{1,1}$ if and only if $u$ is both semiconcave and semiconvex.}
\end{equation}

We will need the following extension result.

\begin{lemma}\label{extdz}
Let $X$ be a finite-dimensional Hilbert space and  $\emptyset \neq P \subset X$ be a bounded
 set, $K\geq 0$ and $c>0$. Let $\psi$ be a Lipschitz function on $P$ such that for each $p \in P$ there exists
 a functional $h_p \in X^*$ such that $|h_p| \leq K$ and
\begin{equation}\label{subgr}
 \psi(p+\Delta)-\psi(p) -  h_p (\Delta)  \leq  c |\Delta|^2\ \ \ \text{whenever}\ \ \ \Delta\neq 0\ \ \ \text{and}\ \ \ p+\Delta \in P. 
\end{equation}
Then there exists a Lipschitz functions $F$ on $X$ which is semiconcave 
%(with linear modulus)
 and $F\restriction_P = \psi$.
\end{lemma}

This result easily follows from a more general result \cite[Proposition 5.12]{DZ} 
 in which the extended function is {\it semiconvex with a general modulus $\vf$}.
It is well-known (see e.g. \cite[p. 30]{CS}) that  $u$ is semiconcave  with a  semiconcavity constant $c$ if and only if $u$ is semiconcave  with modulus $\vf(t)= \frac{c}{2} t$. So, 
 to prove Lemma \ref{extdz}, it is sufficient to apply \cite[Proposition 5.12]{DZ} to $f: -\psi$
 and modulus $\vf(t):= ct$.

More general than  semiconcave functions are DC functions.

\begin{definition}\label{dc}  \rm
 Let $X$ be a Banach spaces, $C\subset X$  an open convex set and
 $Y$ a finite-dimensional Banach space.
\begin{enumerate}
\item[(i)]
A real function on $C$ is called a {\it DC function} if it is a difference of two continuous convex functions.
\item[(ii)]
We say that a mapping $F:C \to Y$ is DC if $y^* \circ F$ is {\it DC} for every functional $y^* \in Y^*$.
%\item[(iii)]
%We say that $f:G \to \R$ (resp. $F:G \to Y$) is locally DC if for each $x \in G$ there exists
% $\delta>0$ such that $f$ (resp. $F$) is DC on $B(x, \delta)$.
\end{enumerate}
\end{definition}
\begin{remark}\label{vldc}
Let $X$, $C$ and  $Y$ be as in Definition \ref{dc}. Then:
\begin{enumerate}
\item[(i)] The system of all DC functions on $C$ is clearly a vector space.
\item[(ii)] $F:C \to Y$ is DC if and only if $y^* \circ F$ is DC for each
   $y^*$ from a basis of $Y^*$. It follows clearly from (i).
\item[(iii)]
If $X$ is a finite-dimensional Hilbert space, then clearly each semiconcave (semiconvex) function on $C$ is DC.
\item[(iv)]	
 If $X$ is a finite-dimensional  Hilbert space and $F: C \to Y$ is $C^{1,1}$, then $F$ is  DC. It follows from \eqref{sscjj}, (iii) and (ii).
%	\item[(v)]
%	If $X$ is finite-dimensional, $F: G \to Y$ is locally DC, $\overline C\subset G$ and  $\overline C$%
%	 is compact, then there exists a Lipschitz DC mapping $\tilde F: X \to Y$ such tha%t
%	 $\tilde F(x) = F(x),\ x \in C$.
	\end{enumerate}
\end{remark}
% easily follows from the well-known facts that each convex function $f$ on $W$ is Lipschitz on each bounded closed convex set  %$K \subset W$ and  $f\restriction_K$ can be extended to a Lipschitz
% convex function on $W$.
%\end{proof}
\begin{definition}\label{plochy}  \rm
	\begin{enumerate}
	\item[(a)]
	 We say that $A \subset \R^d$ is a {\it DC 
	 surface  of dimension $k$} ($0<k<d$), if there exist a $k$-dimensional subspace $W$ of $\R^d$ and a 
	  Lipschitz DC  mapping $\vf: W \to W^{\perp}=:V$ such that  
		  $A = \{w+\vf(w): \ w \in W\}$.
		
		Then we will also say that $A$ is a  DC  surface {\it associated with $V$}.
		\item[(b)]
		For formal reasons, by a  DC  surface of dimension $d$ in $\R^d$ we mean the whole space $\R^d$,
		 and by a DC surface of dimension $k=0$ we mean any singleton in $\R^d$.
		\item[(c)] We say that $A\subset \R^d$ is a {\it semiconcave hypersurface}, if there exist
		 $v \in S_{\R^d}$ and a Lipschitz semiconcave function $g$ on $W:= (\spa\{v\})^{\perp}$ such that
		 $A= \{w + g(w) v:\ w \in W\}$. 
		 \item[(d)]
		We say that $\emptyset \neq A \subset \R^d$ is a  {\it DC (resp. $C^{1,1}$) manifold of dimension $k$}, ($0<k<d$), if for
		 each $a \in A$ there exist  a $k$-dimensional vector space   $W \subset \R^d$, 
		 an open ball $U$ in $W$ and  a DC (resp. $C^{1,1}$) mapping $\vf:U \to W^{\perp}$ such that
		 $P:= \{w + \vf(w):\ w \in U\}$ is a relatively  open subset of $A$ and $a\in P$. 
		\item[(e)] If $k=d$ (or $k=0$), then we mean by a   DC (resp. $C^{1,1}$) manifold of dimension $k$
		 in $\R^d$ a nonempty open set (or an isolated set).
		  \item[(f)]  We say that a $k$-dimensional  ($0<k<d$)  $C^{1,1}$ manifold $ A \subset \R^d$
		   is a {\it uniform $k$-dimen\-sio\-nal  $C^{1,1}$ manifold} if there exists $K>0$ such that
		   (independently on $a$) each $\vf:U \to W^{\perp}$ from (d) can be chosen to be $K$-Lipschitz with
		   $K$-Lipschitz derivative $\vf'$.
		 \end{enumerate}
		 \end{definition}
		 \begin{remark}\label{uzpl}
\begin{enumerate}
\item [(i)]
By Remark \ref{vldc} (iii), each semiconcave hypersurface in $\R^d$ is a $(d-1)$-dimensional DC
 surface.
\item[(ii)]
Using Remark \ref{vldc} (iv), it is easy to see that
each $k$-dimensional $C^{1,1}$ manifold in $\R^d$ is a $k$-dimensional DC manifold.
\item[(iii)]
For a uniform $k$-dimensional  $C^{1,1}$ manifold $A$, the dependence of the tangent space  $\Tan(A,x)$ on $x \in A$ need not 
 be globally Lipschitz, cf.\ Example~\ref{Ex2}, (1).
\end{enumerate}
\end{remark}
		
\section{Basic and auxiliary results on sets of positive reach}\label{bapr}
Given a nonempty set $A\subset\R^d$, we denote by $\Unp A$ the set of all points  $z\in\R^d$ for which the metric projection
\begin{equation}  \label{mproj}
\Pi_A(z):=\{ a\in A:\, \dist(z,A)=|z-a|\}
\end{equation}
is a singleton. Abusing slightly the notation, we shall identify $\Pi_A(z)$ with its unique element in such a case.

If $A\subset \R^d$ and $a\in A$, we define (with $B(a,0):= \emptyset$)
$$\reach(A,a):=\sup\{r\geq 0:\, B(a,r)\subset \Unp A\},$$
and  
$$\reach A:=\inf_{a\in A}\reach (A,a).$$
Obviously, if $\reach A >0$, then $A$ is closed. Further, it is easy to show that
 $\reach A = \infty$ if and only if $A$ is closed convex (cf. \cite[Remark 4.2]{Fe59}).

By $\Tan(A,a)$ we denote the set of all tangent vectors to $A$ at $a$ 
(i.e., $u\in\Tan(A,a)$ if and only if $u=0$ or there exist $a\neq a_i\in A$ and $r_i>0$ such that $a_i\to a$ and $r_i(a_i-a)\to u$, $i\to\infty$) 
which is clearly a closed cone. The normal cone of $A$ at $a$ is defined as the dual cone
$$\Nor (A,a):=\{u\in\R^d:\, \langle u,v\rangle \leq 0\text{ for any }v\in\Tan(A,a)\}.$$
%We further denote the unit normal bundle of $A$ as
%$$\nor A:=\{(a,u):\, a\in A,u\in S_{\R^d}\cap\Nor(A,a)\}.$$

In the following proposition we list some known facts on sets with positive reach.

\begin{proposition}  \label{P_PR}
Assume that $A\subset\R^d$, $a\in A$ and $\reach(A,a)>0$.
\begin{enumerate}
\item[{\rm (i)}] The function $x\mapsto\reach(A,x)$ is either identically equal to $\infty$, or finite and  $1$-Lipschitz on $A$.
\item[{\rm (ii)}] The tangent cone $\Tan(A,a)$ is convex.
\item[{\rm (iii)}] The multifunction $x\mapsto\Nor(A,x)\cap S_{\R^d}$ is upper semicontinuous at $a$.
\item[{\rm (iv)}] If $v\in\Nor(A,a)$ and $b\in A$ then
$$\langle b-a,v\rangle\leq\frac{|b-a|^2|v|}{2\reach(A,a)}.$$
\item[{\rm (v)}] $\Nor(A,a)$ is nontrivial if and only if $a\in\partial A$.
\item[{\rm (vi)}] If $\reach(A,a)>r>0$, then
$$ \Nor(A,a) = \{\lambda v:\ \lambda \geq 0,\, |v|=r, \,\Pi_A(a+v) = a\}.$$ 
\end{enumerate}
\end{proposition}

\begin{proof}
Property (i) follows easily from the definition. For (ii), see \cite[Theorem~4.8.~(12)]{Fe59}. Property (iii) can be deduced from \cite[Theorem~4.8.~(13)]{Fe59} and (iv) follows from \cite[Theorem~4.8.~(7), (12)]{Fe59}. Property (v) (which can be easily deduced from (vi) and (iii)) follows immeditely from Corollary \ref{vnbt}.
For (vi), see \cite[Theorem~4.8.~(12)]{Fe59}.  
\end{proof}

\begin{remark}  \label{LipPr}
If $A\subset\R^d$ is compact and $0<s<\reach A$, then $\Pi_A$ defines a retraction of the open neighbourhood $A_s:=\{z\in\R^d:\, \dist(z,A)<s\}$ onto $A$ (the continuity, even Lipschitzness, of $\Pi_A$ follows from \cite[Theorem~4.8~(8)]{Fe59}). Hence, $A$ is a Euclidean neighbourhood retract and it follows that the fundamental group and the homology groups of $A$ are finitely generated (see, e.g., Corollary~A.8 of \cite{Hat01}). In particular, both $A$ and $\R^d\setminus A$ have only finitely many components.

Using the last property and the Lipschitness of $\Pi_A$ on $A_s$, it is not difficult to prove that $\partial A$ has also finitely many components and any two points lying in the same component of $\partial A$ can be connected {\it in the boundary $\partial A$} by a rectifiable curve.

We note that, however, $\INt A$ may have infinitely many components, see Example~\ref{Ex1}.
\end{remark}

Federer's results of \cite{Fe59} easily imply that
\begin{equation}\label{dime}
\text{if $A \subset \R^d$ is a set of positive reach, then $\dim(A)= \dim_H(A)$.}
\end{equation}
(This fact follows easily from \cite[Remark 4.15 (4), (3), (2)]{Fe59} or, more directly, 
 from Theorem \ref{hlav} below.)

Federer   (\cite[Remark 4.15\, (2)]{Fe59}) also proved that
\begin{equation}\label{dimt}
\text{if $A\subset \R^d$ has positive reach and $a\in A$, then  $\dim(\Tan(A,a)) \leq \dim A$.}
\end{equation}

We will repeatedly use  also the following important Federer's result (\cite[Theorem 4.18]{Fe59}).
\begin{proposition}\label{fedtan}
If $A\subset \R^d$ is a closed set and $0<t<\infty$, then the following two conditions are equivalent:
\begin{enumerate}
\item [{\rm (i)}]
$\reach(A) \geq t.$
\item[{\rm (ii)}]
$\dist(b-a,\Tan(A,a)) \leq |b-a|^2/(2t)$\ \ whenever\ \ $a, \,b\in A$. 
\end{enumerate}
\end{proposition}

\begin{lemma}  \label{L_souv}
Let $A\subset\R^d$ and $\rho>0$ be given.
\begin{enumerate}
\item[(i)] If $B$ is a closed ball in $\R^d$ of radius $\rho$, $A\cap B\neq\emptyset$ and $\reach(A,x)>\rho$ for any $x\in A\cap B$, then $\reach(A\cap B)>\rho$.
\item[(ii)] If $\reach A>\rho$ and $A\neq\emptyset$ is contained in some closed ball of radius $\rho$ then $A$ is contractible (and, hence, connected).
\end{enumerate}
\end{lemma} 

\begin{proof}
Part (i) follows from a more general statement proved in \cite[Lemma~4.3]{Rat02}. Assertion (ii) is a direct consequence of \cite[Remark~4.15~(1)]{Fe59}.
\end{proof}

\begin{lemma}  \label{L_cones}
Assume that $A\subset\R^d$, $a\in A$ and $\reach(A,a)>0$.
Let $C$ be a closed cone contained in $\INt(\Tan(A,a))\cup\{0\}$. Then there exists $r>0$ such that $(a+C)\cap B(a,r)\subset A$.
\end{lemma}

\begin{proof}
Assume, for the contrary, that for any $k\in\N$, there exists $y_k \in ((a+C)\cap B(a,\frac 1k)) \setminus A$. Passing to a subsequence if necessary, we may assume that
$\frac{y_k-a}{|y_k-a|}\to u\in C$. Since $u \in  \Tan(A,a)$, there exist  $z_k \in A$ such that
 $z_k \to a$ and $\frac{z_k-a}{|z_k-a|}\to u$. We can clearly choose $x_k \in [y_k,z_k] \cap \partial A$ for each $k\in\N$. It is easy to show that $x_k \to a$ and $\frac{x_k-a}{|x_k-a|}\to u$.
By Proposition~\ref{P_PR}~(v) we can take a unit normal vector $v_k\in S_{\R^d}\cap \Nor(A,x_k)$, $k\in\N$, and note that, by  Proposition~\ref{P_PR}~(iv),
$$\left\langle v_k,\frac{x_k-a}{|x_k-a|}\right\rangle\geq -\frac{|x_k-a|}{2 \reach(A,x_k)}.$$
Passing to a subsequence, we may achieve that $v_k\to v\in S_{\R^d}$ and Proposition \ref{P_PR} implies  $v \in \Nor(A,a)$. But then, since $\reach(A,x_k)\to\reach(A,a)$, the above inequality implies that $\langle v,u\rangle\geq 0$. But this is a contradiction, since $u$ lies in the interior of $\Tan(A,a)$ and $v$ is in the dual convex cone to $\Tan(A,a)$.
\end{proof}

\begin{corollary}\label{vnbt}
Assume that $A\subset\R^d$, $a\in A$ and $\reach(A,a)>0$. Then
%\begin{enumerate}
%\item [(i)]
$a\in \intr A$ if and only if  $\Tan(A,a) = \R^d$.
%\item  [(ii)] $a \in \partial A $  if and only if  $\dim \Nor(A,a) \geq 1$.
%\end{enumerate}
\end{corollary}

\begin{lemma}\label{hk}
Let   $A\subset\R^d$, $\reach(A)>0$, and  $1\leq k \leq d$. Let  $H^k(A)$ be the set of all $a \in A$ such that
 $\Nor(A,a)$ contains a halfspace of dimension $k$. Then $H^k(A)$ is a closed set.
\end{lemma}
\begin{proof}
Suppose that  $a_i \to a$, where $a_i \in H^k(A)$ for each $i$. Since $A$ is closed, we have $a \in A$. We can clearly for each $i$ find an orthonormal system
 $v_i^1,\dots, v_i^k$ such that all vectors $\pm v_i^1,\dots, \pm v_i^{k-1}$ and $v_i^k$ belong to $\Nor(A,a_i)$.
  Using compactness of $S_{\R^d}$, we can (and will) suppose that $v_i^1\to v^1
  ,\dots, v_i^k\to v^k$. Proposition \ref{P_PR}(iii) easily implies that $\pm v^1,\dots, \pm v^{k-1}$ and $v^k$ belong to $\Nor(A,a)$,
   and therefore $a \in H^k(A)$.
\end{proof}

We will essentially use the following immediate consequence of \cite[Proposition 5.2. and 5.3]{CH}.
\begin{proposition}\label{colhug}
Let   $A\subset\R^d$, $ \reach(A)>r>0$ and $a \in A$. Then the distance function  $d_A:= \dist(\cdot,A)$
is semiconvex 
 on $B(a, r/2)$  with semiconvexity constant $3/r$ and
\begin{equation}\label{clnor}
 \partial d_A(x) = \Nor(A,x) \cap \overline{B(0,1)},\ \ \ x \in A.
\end{equation}
\end{proposition}

We will also need the following simple lemma.

\begin{lemma}\label{prilep}
Let $A\subset \R^d$   and  $\reach(A) > \rho>0$. Let  $a \in A$  and $\Tan(A,a) = \{tu: \ t\geq 0\}$,
 where $|u|=1$. Set $P:= \{a-tu:\ 0<t \leq \rho/4\}$ and  $A^*:= A  \cup P$. Then $\reach(A^*) \geq \rho/4$.
\end{lemma}

\begin{proof}
We can and will suppose $a=0$. Suppose to the contrary that there exists $z \in \R^d$
 with $\dist(z,A^*) < \rho/4$ and two different points $y_1,\, y_2 \in \Pi_{A^*}(z)$.
 The case   $y_1,\, y_2 \in \overline{P}$ is clearly impossible, since $\overline{P}$ 
 is convex. If $y_1,\, y_2 \in A$, then $\dist(z,A) = \dist(z,A^*) < \rho/4$,
 which contradicts  $\reach(A) > \rho>0$. So we can suppose that $y_1\in P$ and
 $y_2 \in A \setminus \{0\}$. Now, if $\langle z,u\rangle \geq 0$, then $|z-0| <|z-y_1|$, a contradiction. So suppose  $\langle z,u\rangle < 0$. Clearly $|y_2|< \rho/2$, which implies
 $\langle y_2,u\rangle >0$. Indeed, if $\langle y_2,u\rangle \leq 0$, then $\dist(y_2, \Tan(A,0))
 = |y_2|$ and so Proposition \ref{fedtan} gives $|y_2| \leq |y_2|^2/ (2\rho)$, a contradiction. Consequently
 there exists $c \in [z,y_2]$ with $\langle c,u \rangle = 0$. Clearly $|c| < \rho/2$
 and $c \in \Nor(A,0)$. So  Proposition \ref{P_PR}(vi) easily gives $|c-y_2|>|c|$. Therefore
$$ |z-0| \leq |z-c| + |c| < |z-c| +  |c-y_2| = |z-y_2|,$$
 which contradicts  $0 \in A^*$.
\end{proof}

\section{Singular points of convex  functions}\label{spcf}

There exists a number of articles which study singularities of convex functions. We will deal with  convex functions $f$ on an open convex set $C \subset \R^d$. Singular points $x$
 of $f$ (i.e., the points of non-differentiability of $f$) are usually classified by the dimension
 of $\partial f(x)$; we use the frequent notation 
$$ \Sigma^k(f):= \{x \in C: \dim \partial f(x) \geq k\}.$$
Then $\Sigma^1(f)$ is the set of all non-differentiability points of $f$. It is well-known 
 for a very long time that $\Sigma^d(f)$ is a countable set. A result of \cite{zajcon} says
 that, for $1\leq k < d$ and $A \subset C$, the set $A$ is contained in   
$ \Sigma^k(f)$
 for some convex $f$ on $C$ if and only if $A$ can be covered by countably many
 DC surfaces of dimension $d-k$ (note that this result is stronger than that of \cite{Al} 
  saying that  $ \Sigma^k(f)$ is a $(\mathcal H^{d-k},d-k) $ set of class $C^2$).
		
Following \cite[p. 82]{CS} we will also consider  sets
$$\Sigma^k_r(f) = \{x \in C:  \partial f(x)\ \ \text{contains a}\ k-\text{dimensional  ball of radius}\ \ r>0\}.$$

Using e.g.	\cite[Theorem 4.1, (4.2)]{AAC}, we easily obtain  that
\begin{equation}\label{schn}
\text{$\Sigma^d_r(f)$ is locally finite
for each  convex function $f$ on $C$ and $r>0$,}
\end{equation}
 which is essentially an easy old result (cf. \cite[p. 14 below]{Sch}).

In this section (see Proposition \ref{konzlomy}) we will show, refining slightly the method of  \cite{zajcon}, 
that $\Sigma^k_r(f)$ can be covered by {\it finitely many}  DC surfaces of dimension $d-k$
 for each Lipschitz convex function $f$ on $C$, $0<k<d$ and $r>0$. It provides a probably  new result,
  except the case $d=2$ and $k=1$, in which it easily follows from \cite[Theorem 3.1]{CP}.
 This result will be essentially applied below and
 can be of some independent interest. (Let us remark that the fact that $\Sigma^k_r(f)$ has locally finite
 $\mathcal H^{d-k}$ measure follows from \cite[Theorem 4.1, (4.2)]{AAC}.)

Although the above mentioned result about sets
 $\Sigma^k_r(f)$ does not hold in infinite-dimensional spaces, we prove the basic Lemmas
 \ref{supo} and \ref{proj}
 in general Banach spaces, since it does not increase the difficulty or length of the exposition and it is possible
 that they can find applications also in this more general context.

\begin{lemma}\label{supo}
Let $E$ be a Banach space, $L>0$, and $\emptyset \neq M \subset E \times \R$. Let the following condition hold. 
\begin{enumerate}
\item[($C_L$)] For each point $m= (e,t) \in M$ there exists  $e_m^* \in E^*$ such that $|e_m^*| \leq L$ and the inequality
 $t+ e_m^*(\tilde e -e) \leq  \tilde t$ holds for every  $\tilde m= (\tilde e,\tilde t) \in M$.
\end{enumerate} 
Then there exists a convex $L$-Lipschitz function $g$ on $E$ such that $M \subset \graph g$.
\end{lemma}
\begin{proof}
For each $m\in M$, choose a corresponding $e_m^*$ and set
$$  g(x):=  \sup \{ t+ e_m^*(x -e):\ m=(e,t) \in M\},\ \ \ x \in E.$$
 Choose  $(e_0,t_0) \in M$.  Since, for each  $x \in E$ and  $m=(e,t) \in M$, 
$$t+ e_m^*(x -e) = t + e_m^* (e_0-e) + e_m^* (x- e_0) \leq t_0 + L|x-e_0|,$$
 $g$ is finite. So, by its definition, $g$ is a convex and $L$-Lipschitz function.
 Using condition ($C_L$), we clearly obtain $M \subset \graph g$.
\end{proof}

\begin{lemma}\label{proj}
Let $X$ be a Banach space, $X= E \oplus K$, where $\dim E \geq 1$ and $1 \leq \dim K< \infty$.
Let $\Omega \subset X$ be an open convex set, $f$ a continuous convex function on $\Omega$, 
  $L>0$ and $\alpha \in K^*$. Set
	$$ A_{\alpha,L} := \{x \in \Omega:\ \alpha = p_x\restriction_K\ \  \text{for some}\ \ 
	 p_x \in \partial f(x) \ \ \text{with}\ \ |p_x|\leq L\}.$$
	 Then there exists a Lipschitz convex function $g$ on $E$ such that
	\begin{equation}\label{projke}
	f(x) - \langle \pi_K(x), \alpha \rangle = g(\pi_E(x))\ \ \ \text{for each}\ \ \ x \in A_{\alpha,L}.
	\end{equation}
\end{lemma}
\begin{proof}
For each $x \in A_{\alpha,L}$, choose a corresponding $p_x \in \partial f(x)$. Further set
$$ M:= \{ (\pi_E(x),f(x) - \langle \pi_K(x), \alpha \rangle):\ \  x \in A_{\alpha,L}\}.$$
 For each $m=(e,t) \in M$ choose  $x_m \in A_{\alpha,L}$ such that
 $e = \pi_E(x_m)$ and $t = f(x_m) - \langle \pi_K(x_m), \alpha \rangle$, and set  $e_m^* := p_{x_m}\restriction_E \in E^*$.

We will show that  the condition ($C_L$) from Lemma \ref{supo} holds. To this end, choose
 arbitrary  $m= (e,t) \in M$, $\tilde m= (\tilde e,\tilde t) \in M$ and set $x:= x_m$, $\tilde x:= x_{\tilde m}$. 
 Since  $p_x \in  \partial f(x)$, we subsequently obtain
\begin{eqnarray*}
f(\tilde x) - f(x) &\geq& \langle \tilde x - x, p_x \rangle = \langle \pi_E (\tilde x - x), e_m^*\rangle + \langle \pi_K (\tilde x - x), \alpha \rangle,\\
f(\tilde x) - \langle \pi_K (\tilde x ), \alpha \rangle  &\geq& f(x) + \langle \pi_E (\tilde x) -
\pi_E ( x), e_m^*\rangle  - \langle \pi_K (x ), \alpha \rangle,\\
\tilde t &\geq& t+ e_m^*(\tilde e -e).
\end{eqnarray*}
Thus the condition ($C_L$) is satisfied and  so by Lemma \ref{supo} there exists a convex $L$-Lipschitz function $g$ on $E$ such that  the inclusion $M \subset \graph g$ (which is clearly equivalent to 
\eqref{projke}) holds.   
\end{proof}
\begin{lemma}\label{koule} Let $X$ be a $d$-dimensional Hilbert space and $K$ its $k$-dimensional
 subspace, $1\leq k <d$. 
%Let $X$ be a Banach space, $X= E \oplus K$, where $\dim E \geq 1$ and $1 \leq k=\dim K< \infty$.
Let $\Omega \subset X$ be an open convex set, $L>0$, $f$ an $L$-Lipschitz convex function on 
$\Omega$, and $\ep>0$.
 For $x \in \Omega$, set  $M_x:= \{p\restriction_K:\ p \in \partial f(x)\} \subset K^*$. Further set
 $$ A_{\ep}^K:= \{x \in \Omega:\ M_x \ \ \text{contains an open ball in}\ \ K^*\ \ \text{with radius}\ \  \ep\}$$
and
$$ Z_{\ep}^K:= \{x \in \Omega:\ f'_+(x,v)+ f'_+(x,-v) > \ep\ \ \ \ \text{whenever}\ \ \ v\in K\ \ \text{and}\ \ |v|=1\}.$$ 
 Then both $A_{\ep}^K$ and $Z
_{\ep}^K$ can be covered by a finite number $N$ of  $(d-k)$-dimensional DC surfaces associated with $K$, where $N=N(k,L, \ep)$ depends only on $k$, $L$ and $\ep$. 

\end{lemma}
\begin{proof}
 Set $E:= K^{\perp}$. 
In the first step we will prove that, for each open ball $B \subset K^*$, the set
 $A_B:= \{x \in \Omega:\ B \subset M_x\}$ is contained in a $(d-k)$-dimensional DC surface associated with $K$. Choose functionals $\alpha_0, \alpha_1, \dots, \alpha_k$  in $B$ such that
 $\alpha_i - \alpha_0$, $i=1,\dots,k$, form a basis of $K^*$. Observe that $|\alpha_i| \leq L$.
 %where $L$ is a Lipschitz constant of $f$. 
So $A_B\subset A_{\alpha_i,L}$ for each $0\leq i \leq k$ (where
 $A_{\alpha_i,L}$ is defined in Lemma \ref{proj}). Using Lemma \ref{proj}, we obtain Lipschitz
 convex functions $g_i$, $i=0,\dots,k$, on $E$ such that
$$f(x) - \langle \pi_K(x), \alpha_i \rangle = g_i(\pi_E(x))\ \ \ \text{for each}\ \ \ x \in A_B.$$
Consequently, for each $x\in A_B$, we have
$$  \langle \pi_K(x), \alpha_i - \alpha_0 \rangle  = g_0(\pi_E(x)) - g_i(\pi_E(x)),\ \ \ i=1,\dots,k.$$
Since $\alpha_i - \alpha_0$, $i=1,\dots,k$, form a basis of $K^*$, it is easy to see (using Remark
 \ref{vldc}(ii))
that $A_B$ is contained in a DC surface associated with $K$.

In the second step observe that  $M_x \subset B^{K^*}(0,2L)$ for each $x\in \Omega$.  Choose a finite $(\ep/2)$-net $Q$ in $ B^{K^*}(0,2L)$ with $N=N(k,L, \ep)$ elements.
 Then clearly each ball of radius $\ep$ in  $ B^{K^*}(0,2L)$ contains a ball $B^{K^*}(q, \ep/2)$ with $q \in Q$.
So  $A_{\ep}^K \subset \bigcup_{q \in Q} A_{B(q,\, \ep/2)}$ and the assertion on $A_{\ep}^K$ follows.

Further observe that  $M_x$ is convex, and   
$f'_+(x,v) =  \sup\{\langle v, \alpha\rangle:\ \alpha \in M_x\}$ for each $x \in \Omega$ and $v \in K$. Consequently, for each $x \in Z_{\ep}^K$ and each unit $v \in K$, we have
\begin{multline*}\sup\{\langle v, \alpha\rangle:\ \alpha \in M_x\} - \inf\{\langle v, \alpha\rangle:\ \alpha \in M_x\} \\
= \sup\{\langle v, \alpha\rangle:\ \alpha \in M_x\} + \sup\{\langle -v, \alpha\rangle:\ \alpha \in M_x\} = f'_+(a,v)+ f'_+(a,-v) > \ep.
\end{multline*}
Thus the minimal width of $M_x$ in the space $K^*$ (which can be identified with $\R^k$) is at least $\ep$, and consequently (see, e.g., \cite{Eg58}) 
 $M_x$ contains a ball of radius  $\ep/(k+1)$.
 So
\begin{equation}\label{zwit} 
Z_{\ep}^K \subset A_{\frac{\ep}{k+1}}^K 
\end{equation}
and the assertion on $Z_{\ep}^K$ follows.
\end{proof}

\begin{proposition}\label{konzlomy}
Let $X$ be a $d$-dimensional Hilbert space, 
%and $K$ its $k$-dimensional subspace,
 $1\leq k <d$ and  $\ep>0$.
%Set $E:= K^{\perp}$.
%Let $X$ be a Banach space, $X= E \oplus K$, where $\dim E \geq 1$ and $1 \leq k=\dim K< \infty$.
Let $\Omega \subset X$ be an open convex set, $L>0$ and  $f$ an $L$-Lipschitz convex function on $\Omega$. 
Denote 
\begin{enumerate}
\item[a)]\ \ by $\Sigma^k_{\ep}$ the set of all $x \in \Omega$ such that $\partial f(x)$ contains
  an open $k$-dimensional ball (i.e., a ball in a $k$-dimensional affine subspace of $X$) of radius $\ep$ and
\item[b)]\ \ by $Z^k_{\ep}$ the set of all $x \in \Omega$ for which there exists a $k$-dimensional
 space $K \subset X$ such that
 $$f'_+(x,v)+ f'_+(x,-v) > \ep\ \ \ \ \text{whenever}\ \ \ v\in K\ \ \text{and}\ \ |v|=1.$$
\end{enumerate}
Then both $\Sigma^k_{\ep}$ and $Z^k_{\ep}$ can be covered by a finite number $N$ of  $(d-k)$-dimensional DC surfaces, where $N=N(d,k,L, \ep)$ depends only on $d$, $k$, $L$ and $\ep$. 
\end{proposition}
\begin{proof}
First we will prove the assertion on $Z^k_{\ep}$.
To this end,  choose $k$-dimensional subspaces $K_1,\dots, K_p$ of $X$ (where $p= p(d,k,L, \ep)$) such that for each  $k$-dimensional space $K$ there exists $1\leq i \leq p$ such that the Hausdorff distance
 of $K \cap S_X$ and $K_i \cap S_X$  is at most  $\frac{\ep}{4L}$. 
 If $x \in Z^k_{\ep}$, choose $K$ by the definition of $Z^k_{\ep}$ and find $K_i$ as above. For each  unit $v \in K_i$ then there exists a unit vector $\tilde v \in K$ such that $|v - \tilde v| < \frac{\ep}{4L}$.
 Since the mapping $w \mapsto f'_+(x,w)$ is $L$-Lipschitz 
(see, e.g., \cite[p. 164, Proposition 7]{Ar}), we easily obtain that
$f'_+(x,v)+ f'_+(x,-v) > \ep/2$.  Thus our assertion follows from Lemma \ref{koule}, since 
$Z^k_{\ep} \subset \bigcup_{i=1}^p  Z^{K_i}_{\ep/2}$.

The assertion on $\Sigma_{\ep}^k$ then follows, since $\Sigma_{\ep}^k\subset Z_{\ep}^k $. 
Indeed, consider an arbitrary $x \in \Sigma_{\ep}^k$ and identify in the usual way $X$ and $X^*$.
 Then there exists a $k$-dimensional space $K\subset X$  and $c\in X$ such that
 $(c+K) \cap B(c,\ep) \subset \partial f(x)$. Consequently, for each unit vector $v\in K$,
\begin{multline*}
 f'_+(x,v)+ f'_+(x,-v) = 
\sup\{\langle v, \alpha\rangle:\ \alpha \in \partial f(x)\} +\sup\{\langle -v, \alpha\rangle:\ \alpha \in \partial f(x)\} \\
 > \langle v, c + \frac{\ep}{2} v\rangle + \langle -v, c - \frac{\ep}{2} v\rangle = \ep.
\end{multline*}
 \end{proof}
As an immediate corollary we obtain the following result.
\begin{corollary}\label{dim1}
Let $\Omega \subset \R^d$ be an open convex set,  $f$ a Lipschitz convex function on $\Omega$,  and $\ep>0$. Let
\begin{eqnarray*}
A_1&:=& \{x \in \Omega: \ f'_+(x,v)+ f'_+(x,-v) > \ep \text{ for some } v \in S_{\R^d}\} \ \text{ and}\\
A_2&:=& \{x \in \Omega: \ \diam(\partial f(x)) > \ep\}.
\end{eqnarray*}
Then both $A_1$ and $A_2$ can be covered by finitely many  DC hypersurfaces.
\end{corollary}
%\begin{proof}
%The desired property of $A_1$ immediately follows from Proposition \ref{konzlomy}.
%If $x \in A_2$, then we can find $\alpha_1$, $\alpha_2$ in $\partial f(x)$ and a unit $v \in \R^n$
% such that $\langle v, \alpha_2- \alpha_1 \rangle > \ep$. Then
%$$  f'_+(x,v)+ f'_+(x,-v) \geq \langle v, \alpha_2 \rangle  + \langle -v, \alpha_1 \rangle >\ep.$$
% So $A_2 \subset A_1$ and we are done.
%\end{proof}

\section{Singular points of sets with positive reach}\label{sppr}

If $A \subset \R^d$ is a set with positive reach  and $0\leq k \leq d$,
 we set, following  Federer (\cite[p. 447]{Fe59}),
$$ A^{(k)}:= \{a\in A: \dim\, \Nor (A,a) \geq d-k\}.$$
(The points of $A^{(k)}$ are, for $k\neq d$, sometimes called $k$-singular boundary points of $A$
 and the symbol $\Sigma^k(A)$ is then used instead of   $A^{(k)}$; see, e.g., \cite{Hu}. However,
 we will use Federer's notation.)

Federer proved (\cite[p. 447]{Fe59}) that $ A^{(k)}$ is countably $k$ rectifiable.  
 Using Proposition \ref{colhug} and  results of \cite{zajcon}, it is easy to obtain a stronger result. (It will be obtained below
  as a consequence of more subtle Proposition \ref{poakep}; see Remark \ref{hug}.)

\begin{proposition}\label{sipr}
If $A \subset \R^d$ is a set with positive reach  and $1\leq k \leq d-1$, then
$A^{(k)}$ can be covered by coutably many $k$-dimensional DC surfaces.
\end{proposition}

\begin{remark}\label{dvezle}
We will improve Proposition \ref{sipr} as follows:
\begin{enumerate}
\item [(i)] \  $A^{(d-1)}= \partial A$ can be locally covered by {\it finitely many} $(d-1)$-dimensional {\it semiconcave} surfaces  (Theorem \ref{pokrhr}).
\item [(ii)]\
\ If $1\leq k=\dim A$, then 
\begin{enumerate}
\item\ $A^{(k-1)}$ can be  locally covered by {\it finitely many} $(k-1)$-dimensional DC surfaces (Proposition \ref{zlfed}) and
\item $A^{(k)}$ can be  locally covered by {\it finitely many} $k$-dimensional DC surfaces (Theorem \ref{hlav}).
\end{enumerate}
\end{enumerate}

To prove (ii), we will classify the points of 
 $A^{(k)}$ by a ``strength of singularity'' in a similar (but different) way as in \cite{Hu}
 (cf. Remark \ref{hug} below):
\end{remark}

\begin{definition}\label{akep}  \rm
If $A \subset \R^d$ is a set with positive reach, $0\leq k<d$ and $\ep>0$, then we denote
 by  $A^{(k)}_{\ep}$ the set of all points of $A$, for which $\Nor (A,a)\cap B(0,1)$ contains a
 $(d-k)$-dimensional ball of radius $\ep$.
\end{definition}

\begin{proposition}\label{poakep}
Let $A \subset \R^d$ be a  set with positive reach, $0\leq k<d$ and $\ep>0$.
 Then $A^{(k)}_{\ep}$ can be locally covered by   finitely many $k$-dimensional  DC surfaces.

 Moreover, if $k>0$, $\reach A> r >0$ and $a \in A$, then the set  $A^{(k)}_{\ep}\cap B(a,r/2)$
  can be covered by finite number $N$ of $k$-dimensional  DC surfaces, where
   $N= N(d,k,\ep)$ depends only on $d$, $k$ and $\ep$.
\end{proposition}
\begin{proof}
Consider an arbitrary $a \in A$ and denote $d_A:= \dist(\cdot,A)$.  By Proposition \ref{colhug} the function $g(x):=d_A(x) + 3(2r)^{-1} |x|^2,\  x\in  B(a, r/2)$, is convex, which clearly implies that also the function
$$f(x):=d_A(x) + 3 (2r)^{-1} |x-a|^2,\ \ x\in  B(a, r/2)$$
 is convex.  Moreover, since $d_A$ is $1$-Lipschitz, it is easy to see that $f$ is Lipschitz with
  constant $1 + 3/2 = 5/2$.
By the basic properties of Clarke subgradient (see \cite[Corollary 1 of Proposition 2.3.3]{Cl90}) we have
 $\partial f(x) = \partial d_A(x) + (3/r) (x-a)$ for each $ x \in B(a, r/2)$.
 Hence Proposition \ref{colhug} implies that, for each $x \in A^{(k)}_{\ep} \cap  B(a, r/2)$, $\partial f(x)$
  contains a $(d-k)$-dimensional ball of radius $\ep$. So, if $k>0$,  Proposition \ref{konzlomy}
	 implies that 
	 $A^{(k)}_{\ep} \cap  B(a, r/2)$  can be covered by a finite number $N$ of $k$-dimensional  DC surfaces, where
   $N= N(d,k,\ep)$. If $k=0$, then \eqref{schn} easily gives that $A^{(k)}_{\ep}$ is locally finite.
\end{proof}

\begin{corollary}\label{copo}
Let $A \subset \R^d$ be a compact set with positive reach, $0\leq k<d$ and $\ep>0$.
 Then $A^{(k)}_{\ep}$ can be  covered by   finitely many $k$-dimensional DC surfaces.
\end{corollary}

\begin{remark}\label{hug}
 Since clearly  $A^{(k)} = \bigcup_{i=1}^{\infty} A^{(k)}_{1/i}$, Proposition \ref{poakep}
 implies Proposition \ref{sipr}.
\end{remark}

\begin{remark}\label{hug2}
Proposition \ref{poakep} is closely related to a result of \cite{Hu}.
Namely, Hug (\cite[p. 2]{Hu}) considers (for $0\leq k <d$) sets
$$ \Sigma^k(A,\ep):= \{x \in \partial A:\ \mathcal H^{d-1-k}(\Nor(A,x) \cap S_{\R^d}) \geq \ep\},$$
 which are closely related to our sets  $A^{(k)}_{\ep}$. Namely, it is not difficult to show
 that each set $\Sigma^k(A,\ep_1)$ is contained in some  $A^{(k)}_{\ep_2}$ and vice versa.
 Consequently, \cite[Theorem~3.2]{Hu} implies (cf.\ also \cite[Corollary~3.6]{Hu}) that $A^{(k)}_{\ep}$ has locally finite
 $\mathcal H^{k}$ measure and Proposition~\ref{poakep} implies that $\Sigma^k(A,\ep)$
 can be locally covered by   finitely many $k$-dimensional  DC surfaces.
\end{remark}

Suppose that $A \subset \R^d$ is a set with  positive reach. Federer (\cite[p. 447]{Fe59}) proved the following interesting result:
\begin{multline}\label{tanvec}
\text{ if $\dim(A)=k\geq 1$, then $A= A^{(k)} \neq A^{(k-1)}$ and, for $a \in A \setminus  A^{(k-1)}$},\\
 \text{$\Tan(A,a)$ is a $k$-dimensional vector space.}
\end{multline}
Moreover, Federer claimed without a proof 
(\cite[Remark 4.20]{Fe59}) that, if $\dim(A)=k\geq 1$, then
\begin{equation}\label{kmjc}
 \text{ $A^{(k-1)}$ is closed, and}
\end{equation}
\begin{equation}\label{fecjj}
\text{ $A \setminus  A^{(k-1)}$ is a $k$-dimensional
 $C^{1,1}$ manifold.}
\end{equation}
We will prove  \eqref{kmjc} in the following proposition, and statement \eqref{fecjj} will be proved
 in Theorem \ref{hlav} below. 

\begin{proposition}\label{zlfed}
Let $A \subset \R^d$ be a set with positive reach and $\dim(A)=k\geq 1$. Then
\begin{enumerate}
\item[(i)]
 $A^{(k-1)}$ is closed and 
\item[(ii)]
 $A^{(k-1)}$ can be locally covered by finitely many  DC surfaces
 of dimension $k-1$.

 Moreover, if $k>1$, $\reach A> r >0$ and $a \in A$, then the set  $A^{(k-1)}\cap B(a,r/2)$
  can be covered by a finite number $N$ of $(k-1)$-dimensional DC surfaces, where
   $N= N(d,k)$ depends only on $d$ and $k$.
\end{enumerate}
\end{proposition}  
\begin{proof}
Let $a \in A^{(k-1)}$. By definition of $A^{(k-1)}$,   $\Nor (A,a)$ is a closed convex cone of dimension
 at least $d-k+1$.  Since  $\dim \Tan(A,a) \leq k$  by \eqref{dimt},   we obtain that $\Nor (A,a)$ contains a vector space of dimension  $d-k$. Thus 
$\Nor (A,a)$ clearly contains a halfspace  of dimension $d-k+1$. Consequently we obtain
 that $A^{(k-1)} = H^{d-k+1}(A)$, (where $H^{d-k+1}(A)$ is defined in Lemma \ref{hk}). Thus (i)
 follows from Lemma \ref{hk}.

Further, it is easy to see that $H^{d-k+1}(A) \subset A^{(k-1)}_{1/2}$, and so (ii)
 follows from Proposition \ref{poakep}.
\end{proof}

\begin{theorem}\label{pokrhr}
 Let $A \subset \R^d$ be a set with $\reach(A) > r>0$ and $a \in \partial A$. Then there exists
 a finite system $\cal S$ of semiconcave hypersurfaces
 which covers $B(a,r/2) \cap \partial A$.

Moreover, $\card \cal S = N$, where $N= N(d)$ depends only on $d$.
\end{theorem}
\begin{proof}
Choose a finite $(1/4)$-net $F$ in $S_{\R^d}$ with the cardinality $N=N(d)$. For each $v \in F$ set
$$ M_v:= \{ z \in  \partial A:\ |v-n_z|< 1/4\ \ \ \text{for some unit vector}\ \ \ n_z\in \Nor(A,z)\}.$$
  By Proposition \ref{P_PR} (v), $\partial A = \bigcup_{v \in F} M_v$, and so it is sufficient
	 to show that
	for each $v \in F$, the set $S:=B(a,r/2) \cap M_v$ is  a subset of a semiconcave hypersurface.
	 To this end, fix an arbitrary $v \in F$, and for each $z \in M_v$, choose some $n_z$ from
	 the definition of $M_v$. Denote  $V:= \spa \{v\}$ and $W:= V^{\perp}$.
	  Observe that if $x\in S$, then
		\begin{equation}\label{klss}
		\langle v, n_x \rangle = \langle v, v\rangle + \langle v,n_x- v\rangle \geq 3/4 >0.
		\end{equation}
		Without any loss of generality, we can suppose that
$a=0$ and  $v= e_d$. We will identify $W= \spa\{e_1,\dots,e_{d-1}\}$ with $\R^{d-1}$.
 Now consider  two arbitrary points $x \in S$, $y\in S$. Using Proposition \ref{P_PR}(iv)
	 and $|x-y| < r$, we obtain
	\begin{equation}\label{stner}
		\langle y-x, n_x\rangle  \leq \frac{|y-x|^2}{2 r} \leq \frac{|y-x|}{2}.
		\end{equation}
		Writting $y-x = w_1 + v_1$ with $w_1\in W$ and $v_1\in V$, \eqref{stner} and  $|v-n_x|<1/4$ yield
\begin{multline*}
	|v_1|= |\langle y-x,v \rangle| = |\langle y-x,n_x\rangle + \langle y-x, v - n_x\rangle| \\
	 \leq |y-x|/2 + |y-x|/4 \leq (3/4) |v_1| + (3/4) |w_1|,
	\end{multline*}
	which immediately implies $|v_1| \leq 3 |w_1|$.
Consequently  
	 $S$ is the graph of a $3$-Lipschitz function  $\psi$
	 defined on a set $P \subset \R^{d-1}$.
	
	Now fix an arbitrary $p \in P$, denote $x:=(p, \psi(p))$ and define $h_p \in W^*= (\R^{d-1})^*$
	 putting  $h_p(u):= -\langle u,n_x\rangle/\langle v,n_x\rangle$ for $u \in W$.
	
	Using $|n_x|=1$ and \eqref{klss}, we see that $|h_p| \leq 4/3$. Further observe
	 that the graph of $h_p$ is orthogonal to $n_x$:
	\begin{equation}\label{kolm}
	\langle n_x, \Delta + h_p(\Delta) v\rangle = \left\langle n_x, \Delta - \frac{\langle \Delta,n_x\rangle}{\langle v,n_x\rangle}\, v    \right\rangle =0\ \ \text{for each}\ \ \Delta \in W.
	\end{equation}
	Now we will verify the condition \eqref{subgr} from Lemma \ref{extdz}. To this end, consider
 $p$ and $x$ as above   and an arbitrary $\Delta \in \R^{d-1}$ such that $p+ \Delta \in P$. 
Set  $\omega:= 	\psi(p+\Delta)- \psi(p) - h_p (\Delta)$.
Since  \eqref{subgr} is trivial for $\omega \leq 0$, we suppose $\omega>0$.
 Denote  
$$ y:= (p+\Delta, \psi(p+\Delta)),\ \  z:= 	(p+\Delta, \psi(p) + h_p (\Delta)).$$
Then 
$ y-z = \omega\, v$ and therefore $\omega = |y-z|$. Further
$\langle n_x, z-x\rangle = \langle n_x, \Delta + h_p(\Delta) v   \rangle = 0$ by \eqref{kolm}.
Using also 
 \eqref{klss} we obtain
\begin{equation}\label{ctver}
 \langle n_x, y-x\rangle = \langle n_x, z-x\rangle + \langle n_x, y-z\rangle = \langle n_x, y-z\rangle  
= \langle n_x, \omega\,  v\rangle > 0.
\end{equation}
Further
\begin{multline*}
|y-z| = |\langle v, y-z \rangle| \leq |\langle n_x, y-z \rangle| + |\langle v - n_x, y-z \rangle|
\leq |\langle n_x, y-z \rangle| + \frac{1}{4} |y-z|.
\end{multline*}
Hence, using also \eqref{ctver} and \eqref{stner}, we obtain
$$ \omega = |y-z| \leq \frac 43 |\langle n_x, y-z \rangle| = \frac 43 \langle n_x, y-x \rangle
 \leq \frac 43 \frac{|y-x|^2}{2 r}.$$
So, since $3$-Lipschitzness of $\psi$ gives $|y-x| = |(\Delta, \psi(p+\Delta)- \psi(p))| \leq
 4|\Delta|$, we obtain
$$
\psi(p+\Delta)- \psi(p) - h_p (\Delta) = \omega \leq \frac{32}{3}\, \frac{|\Delta|^2}{r}
=: c |\Delta|^2.
$$
So \eqref{subgr} holds and thus Lemma \ref{extdz} gives that $S$ is  a subset of a semiconcave hypersurface.
\end{proof}

\section{Sets of positive reach in the plane}\label{plane}

We start with two lemmas which will be needed later.

\begin{lemma}  \label{L1_pom}
Let $\delta,\rho>0$ and $0<\eta<1$ be such that 
$$\rho\eta >\delta.$$
Let further $A\subset\R^2$ and a vertical segment $S\subset\R^2$ of length less or equal to $2\delta$ be given. Assume that for any $x\in A\cap S$, $\reach(A,x)\geq\rho$ and
\begin{equation}   \label{L1E}
|\langle v,e_2\rangle|\geq\eta |v|\text{ whenever }v\in\Nor(A,x).
\end{equation}
Then, the intersection $A\cap S$ is connected.
\end{lemma}

\begin{proof}
First, note that the assumption $\reach (A,x)\geq\rho$, $x\in A\cap S$, implies that $A\cap S$ is closed.
Assume, for the contrary, that $A\cap S$ is not connected, i.e., there exist two points $x=(x_1,x_2)$, $y=(x_1,y_2)$ in $S\cap A$ with $x_2<y_2$ and such that the open segment $(x,y)$ does not intersect $A$. Then, $x\in\partial A$ and we claim that there exists a vector $v\in\Nor(A,x)$ with $\langle v,e_2\rangle>0$. Indeed, if not, \eqref{L1E} would imply that $e_2$ lies in the interior of $\Tan(A,x)$, and Lemma~\ref{L_cones} would imply that $x+\tau e_2\in A$ for sufficiently small $\tau>0$, which would contradict our assumption. So, let $v=(v_1,v_2)\in \Nor(A,x)$ be a unit vector with $v_2>0$. \eqref{L1E} implies that $v_2\geq\eta$ and Proposition~\ref{P_PR}~(iv) yields $\langle y-x,v\rangle\leq|y-x|^2/(2\rho)$, hence, $v_2\leq 2\delta/(2\rho)$. Putting these estimates of $v_2$  together, we obtain $\eta\leq\delta/\rho$, which contradicts our assumption and completes the proof.
\end{proof}

\begin{lemma}  \label{L2_pom}
Let $\varphi:I\to\R$ be a function defined on an interval $I\subset \R$ and  $A\supset\graph\varphi$.  
 Let
$\delta,\rho>0$ and $0<\eta<1$ be such that
$2\delta<\rho \eta$, $\diam(\graph\varphi)\leq 2\delta$ and
for any $x\in\graph\varphi$ we have $\reach(A,x)>\rho$ and
\begin{equation} \label{L2E}
\exists v\in\Nor(A,x)\cap S_{\R^2}:\, \langle v,e_2\rangle\geq\eta .
\end{equation}
Then $\varphi$ is Lipschitz.
\end{lemma}

\begin{proof}
Consider two different numbers $s_1,s_2\in I$ and denote
$$x_1=(s_1,\varphi(s_1)),\, x_2=(s_2,\varphi(s_2)),\ u=(u_1,u_2):=\frac{x_2-x_1}{|x_2-x_1|}.$$
To prove the Lipschitz property of $\varphi$, it is clearly sufficient to prove
 that $u_2\leq \lambda$
for some constant $\lambda <1$ (independent of $s_1,s_2$). 
If $v=(v_1,v_2)\in\Nor(A,x_1)$ is a unit vector from \eqref{L2E} then by Proposition~\ref{P_PR}(iv) and since $|x_2-x_1|\leq 2\delta$, we get
\begin{equation} \label{Eq03}
\langle u,v\rangle\leq \frac{|x_2-x_1|}{2\rho}\leq \frac{\delta}{\rho}.
\end{equation}
Observing that clearly
$u_2v_2\leq \langle u,v\rangle+|u_1v_1|$, and using $v_2\geq\eta$ and \eqref{Eq03}, we obtain
$$u_2\leq\frac{\delta}{\rho}\frac 1{\eta}+\frac{|u_1|}{\eta}\leq\frac 12+\frac{|u_1|}{\eta}.$$
If $|u_1|\leq\frac{\eta}{4}$ then $u_2\leq\frac 34$, and if not then $|u_2|=\sqrt{1-u_1^2}\leq\frac{\sqrt{16-\eta^2}}4<1$. So it is sufficient to put $\lambda:= \max\{\frac 34, \frac{\sqrt{16-\eta^2}}4\}$. 
\end{proof}

\begin{definition}  \rm   \label{D_types}
Let $M \subset \R^2$ and $r>0$. We say that
\begin{enumerate}
\item
$M$ is a {\it $\tilde T_r^1$-set} if there exists a Lipschitz semiconcave function $\vf$ on $(-r,r)$
 such that $\vf(0)=0$ and  $M= B(0, r) \cap \hypo \vf$.
\item
$M$ is a {\it $\tilde T_r^2$-set} if there exist  Lipschitz  functions $\psi \leq \vf$ on $(-r,r)$
 such that    $\vf$ is semiconcave, $\psi$ is semiconvex,  $\vf(0)=\psi(0) = 0$,
   $\vf'(0)=\psi'(0) = 0$ and      $M= B(0, r) \cap \hypo \vf \cap \epi \psi$.
\item
$M$ is a {\it $\tilde T_r^3$-set} if there exist  Lipschitz  functions $\psi \leq \vf$ on $[0,r)$
 such that    $\vf$ is semiconcave on $(0,r)$, $\psi$ is semiconvex on $(0,r)$,  $\vf(0)=\psi(0) = 0$,
   $\vf'_+(0)=\psi'_+(0) = 0$ and      $M= B(0, r) \cap \hypo \vf \cap \epi \psi$.
	\item
	$M$ is {\it of type $T^i$} ($i=1,2,3$) {\it at} $x\in M$, if there exists an isometry $G: \R^2 \to \R^2$
		 such that  $G(x)=0$ and $G(M \cap  B(x, r))$ is a $\tilde T_r^i$-set for some $r>0$.
\end{enumerate}
\end{definition}

\begin{theorem}  \label{T_PR-char}
Let $A\subset\R^2$ and $a\in A$ be given. Then $\reach(A,a)>0$ if and only if one of the following statements holds:
\begin{enumerate}
\item $a$ is an interior point of $A$,
\item $a$ is an isolated point of $A$,
\item $A$ is of type $T^i$ at $a$ for some $i\in\{1,2,3\}$.
\end{enumerate}
\end{theorem}

\begin{proof}
Clearly, $\reach (A,a)>0$ if $a$ is an interior or isolated point of $A$. We shall show that the same is true under (3). 

Assume that $A$ is of type $T^i$ at $a$ ($i=1,2,3$). We can assume without loss of generality that $a=0$ and that $A\cap B(0,r)$ is a $\tilde{T}^i_r$-set for some $r>0$. Consider first the case $i=1$ and let $\varphi$ be the Lipschitz semiconcave function from Definition~\ref{D_types}. By using \cite[Proposition~1.7]{Fu85}, we can consider $\varphi$ to be defined, Lipschitz and semiconcave on the whole $\R$, and \cite[Theorem~2.3]{Fu85} implies that $\reach(\hypo\varphi)>0$ (cf.\ (A) in Introduction). Since $A$ coincides with $\hypo\varphi$ on a neighbourhood of $0$, we infer that $\reach (A,0)>0$. 

Let now $i=2$ and let $\varphi$ and $\psi$ be as in Definition~\ref{D_types}. Again, we can assume $\varphi,\psi$ to be defined on $\R$, Lipschitz and semiconcave, semiconvex, respectively ($\psi\leq\varphi$ on $(-r,r)$). Set $\rho:=\min\{\frac r2,\reach(\hypo\varphi,0),\reach(\epi\psi,0)\}$ and take a point $x=(x_1,x_2)\in B(0,\rho)$. We distinguish three cases: If $x_2>\varphi(x_1)$ then $\Pi_A(x)=\Pi_{\hypo\varphi}(x)$, if $x_2<\psi(x_1)$ then $\Pi_A(x)=\Pi_{\epi\psi}(x)$, and if $\psi(x_1)\leq x_2\leq\varphi(x_1)$ then $\Pi_A(x)=x$. In all these cases, the metric projection to $A$ is single-valued at $x$ and, hence, $\reach(A,0)\geq\rho$.

Assume now $i=3$, let $\varphi$ and $\psi$ be as in Definition~\ref{D_types} and, again, assume that $\varphi,\psi$ are defined on $\R$. 
Decreasing $r>0$ if necessary, we can assume that $A\cap B(0,r)$ is contained in the cone $\{x:\, \langle x,e_1\rangle\geq\frac{\sqrt{3}}{2}|x|\}$. If $x\in B(0,\frac r2)$ lies in the dual cone, $\{x:\, \langle x,e_1\rangle\leq -\frac 12|x|\}$, then, clearly, $\Pi_A(x)=0$. Consider the functions 
$$\tilde{\varphi}(x):=\min\{x,\varphi(x)\},\quad \tilde{\psi}(x):=\max\{-x,\psi(x)\}.$$ 
The function $\tilde{\varphi}$ ($\tilde{\psi}$) is clearly semiconcave (semiconvex) and coincides with $\varphi$ ($\psi$, respectively) on $(0,r)$ (see, e.g.,
 \cite[Proposition 2.1.5]{CS}). Set 
$$\rho:=\min\{\frac r2,\reach(\hypo\tilde{\varphi}), \reach(\hypo\tilde{\psi})\}$$
 and consider a point $x=(x_1,x_2)\in B(0,\rho)$ such that $x_1=\langle x,e_1\rangle\geq -\frac 12|x|$. 
Again, we distinguish three cases.
If $x_2\geq \tilde{\varphi}(x_1)$ then $\Pi_A(x)=\Pi_{\hypo\tilde{\varphi}}(x)$ is a singleton. If $x_2\leq \tilde{\psi}(x_1)$ then $\Pi_A(x)=\Pi_{\epi\tilde{\psi}}(x)$ is again a singleton. 
If $\tilde{\psi}(x_1) < x_2 < \tilde{\varphi}(x_1)$, then clearly $x_1>0$ and therefore $\Pi_A(x)=x$.
Thus, $\reach(A,0)\geq\rho$.

We shall show the other implication. Assume that $a\in A$ and 
$$r_0:=\min\{\reach(A,a),1\}>0,$$ 
and let $a$ be neither an interior, nor an isolated point of $A$. Then, $\Tan(A,a)$ is a convex cone that neither reduces to $\{0\}$, nor equals the whole $\R^2$ (since then, by Corollary \ref{vnbt},
 $a$ would be an interior point of $A$).
We shall distinguish three cases.

(i) Let $\Tan(A,a)$ be two-dimensional, i.e., there exists a unit vector $v_0$ and an $\eta'\in (0,1]$ such that
$$u\in\Tan(A,a)\iff \langle u,v_0\rangle\leq -\sqrt{1-\eta'^2} |u|.$$
We can assume without loss of generality that $a=0$ and $v_0=e_2$. We have then
$\Nor(A,0)=\{v:\, \langle v,e_2\rangle\geq\eta'|v|\}$.
Using the definition of the tangent cone, Lemma~\ref{L_cones} and Proposition~\ref{P_PR}~(iii), (i), subsequently, fixing any $0<\eta<\eta'$, we can find a $0<\delta<r_0\eta/4$ such that
\begin{eqnarray}
&&A\cap B(0,\delta)\subset\{x:\, \langle x,e_2\rangle\leq \sqrt{1-\eta^2}|x|\}, \label{Eq0}\\
&&B(0,\delta)\cap\{x:\, \langle x,e_2\rangle\leq -\sqrt{1-\eta^2}|x|\}\subset A,  \label{Eq1}\\
&&\langle v,e_2\rangle\geq \eta |v|\text{ whenever }x\in A\cap \overline{B}(0,\delta)\text{ and }v\in\Nor(A,x),  \label{Eq2} \\
&&\reach(A,x)>\frac{r_0}2 \text{ whenever }x\in A\cap \overline{B}(0,\delta).  \label{Eq3}
\end{eqnarray}
We shall use the notation for vertical lines 
$$\ell(s):=\{ x\in\R^2:\, \langle x,e_1\rangle =s\},\quad s\in\R.$$
Lemma~\ref{L1_pom} (with $\rho=r_0/2$ and $S=\ell(s)\cap\overline{B}(0,\delta)$), \eqref{Eq2} and
\eqref{Eq3} yield that 
\begin{equation}\label{prjsou}
\text{$A\cap\ell(s)\cap B(0,\delta)$ is connected whenever $|s|<\delta$.}
\end{equation}
If $|s|<\eta\delta$, then $s \sqrt{1-\eta^2}/\eta < \sqrt{\delta^2-s^2}$ and an elementary computation shows that \eqref{Eq0} and \eqref{Eq1} imply
\begin{equation}\label{dveink}
\{s\} \times (- \sqrt{\delta^2-s^2},-s\eta^*] \subset
  A\cap\ell(s)\cap B(0,\delta) \subset \{s\} \times (- \sqrt{\delta^2-s^2}, 
	s\eta^*],
\end{equation}
where $\eta^*:=\sqrt{1-\eta^2}/\eta$.
So, fixing any $0< r <\eta\delta /4$, we obtain that the function
\begin{equation}  \label{def_phi}
\varphi(s):=\sup\{t\in\R:\, (s,t)\in A\cap\ell(s)\cap B(0,\delta)\},\, |s|< 4r,
\end{equation}
 is finite and, for each $s \in (-4r, 4r)$,
\begin{equation}\label{odhvf}
 - \sqrt{\delta^2-s^2} < - s\eta^* \leq \vf(s) \leq s\eta^* < \sqrt{\delta^2-s^2},
\end{equation}
which clearly implies that the graph of $\varphi$ is contained in $\partial A\cap B(0,\delta)$.
Consequently, due to Proposition~\ref{P_PR}~(v), \eqref{Eq2} and \eqref{Eq3}, we can apply Lemma~\ref{L2_pom} (with $\rho=r_0/2$ and $I=(-4r,4r)$) and get that $\varphi$ is Lipschitz. Set  $V_r:=\{(x_1,x_2):\, |x_1|<4r\}$. Clearly,
\eqref{prjsou}  and \eqref{dveink}
 imply that
\begin{equation}\label{stop}
A\cap B(0,\delta)\cap V_r=\hypo\varphi\cap B(0,\delta).
\end{equation}
Assume now that $|s|<r$ and  $x=(s,\varphi(s))$. We know that $x \in \partial A\cap B(0,\delta)$ and so
 there exists (see Proposition~\ref{P_PR}~(v)) a vector $v\in\Nor(A,x)\cap S_{\R^2}$. Proposition~\ref{P_PR}~(vi), \eqref{Eq3} and $r<r_0/2$  imply that $B(x+ rv, r)\cap A=\emptyset$.
 It is clear that  $B(x+rv,r)\subset  V_r$. Using $|s| < r< \delta \eta/4$ and \eqref{odhvf},
 we obtain $|\vf(s)| < \delta/4$ and consequently $|x|\leq |s| + |\vf(s)|< \delta/2$. Hence, since $r< \delta/4$,
 clearly  $B(x+ rv, r) \subset B(0,\delta)$.
Thus $B(x+ rv, r)\cap A=\emptyset$ and \eqref{stop} imply
$B(x+rv,r)\cap\hypo\varphi=B(x+rv,r)\cap A=\emptyset$
and so, by \cite[Theorem~2.6]{Fu85}, $\varphi$ is semiconcave on $(-r,r)$. Hence, using also that $\varphi(0)=0$, we get that $A\cap B(0,r)$ is a $\tilde{T}_r^1$-set and thus $A$ is of type $T^1$ at $0$.

(ii) Assume now that $\Tan(A,a)$ is a line; without loss of generality we assume that it is the $x_1$-axis and, again, that $a=0$. Hence, $\Nor(A,0)$ is the $x_2$-axis and, using the definition of the tangent cone and Proposition~\ref{P_PR}~(iii), we see that for any fixed $\eta\in(0,1)$ there exists a $0<\delta<r_0\eta/4$ such that \eqref{Eq3} holds and
\begin{eqnarray}
&&A\cap B(0,\delta)\subset\{x:\, |\langle x,e_2\rangle|\leq \sqrt{1-\eta^2}|x|\},  \label{Eq4} \\
&&|\langle v,e_2\rangle|\geq \eta |v|\text{ whenever }x\in A\cap \overline{B}(0,\delta)\text{ and }v\in\Nor(A,x).  \label{Eq5}
\end{eqnarray} 
Lemma~\ref{L_souv} and \eqref{Eq3} yield that $A\cap\overline{B}(0,\delta/2)$ is connected. This implies that also $\Pi_1(A\cap \overline{B}(0,\delta/2))$ is connected, where $\Pi_1$ denotes the orthogonal projection to the $x_1$-axis. Since both $e_1,-e_1$ are tangent vectors of $A$ at the origin, $\Pi_1(A\cap \overline{B}(0,\delta/2))$ must contain a neighbourhood of the origin in $\R$ and, so, we can choose an $0<r<\eta\delta/4$ such that 
$(-4r,4r)\subset\Pi_1(A\cap \overline{B}(0,\delta/2))\subset\Pi_1(A\cap B(0,\delta))$. Thus,
$A\cap\ell(s)\cap B(0,\delta)$ is nonempty if $|s|<4r$.

%Assume that this is not the case for some $|s|<\eta\delta /2$. Let $x$ be some closest point of $A\cap B(0,\delta)$ to the line $\ell(s)$ (such a point exists since $A$ is relatively closed in $B(0,\delta)$ and the distance from $\ell(s)$ to the origin is smaller than that to the boundary of $B(0,\delta)$). But then, there must be a normal vector to $A$ at $x$ perpendicular to $e_2$, which contradicts \eqref{Eq5}.

Using Lemma~\ref{L1_pom} (again with $\rho=r_0/2$ and $S=\ell(s)\cap\overline{B}(0,\delta)$),
\eqref{Eq3} and \eqref{Eq5}, we find that the intersection $A\cap\ell(s)\cap B(0,\delta)$ is connected for any $|s|< \delta$. We define the function $\varphi$ on $(-4r,4r)$ again by \eqref{def_phi}. Thus, using \eqref{Eq4}, we obtain, similarly as in the case (i), that
 \eqref{odhvf} holds again, and consequently 
 we obtain $\graph\varphi\subset\partial A\cap B(0,\delta)$ again.

We claim that at any point $x\in\graph\varphi$ there exists a unit vector $v\in\Nor(A,x)$ with  
$\langle v,e_2\rangle\geq \eta$. (Indeed, assume that this is not the case; then, due to \eqref{Eq5}, all normal vectors $u$ to $A$ at $x$ satisfy $\langle u,e_2\rangle\leq -\eta|u|$ and, hence, $e_2$ has to be in the interior of $\Tan(A,x)$. But then, using Lemma~\ref{L_cones}, we get that $[x,x+\ep e_2]\subset A$ for some small $\ep>0$, which contradicts the definition of $\varphi$.) Thus,  we may apply Lemma~\ref{L2_pom} (again with $\rho=r_0/2$ and $I=(-4r,4r)$) and \eqref{Eq3}, and  get the Lipschitz property of $\varphi$. 

We define also 
$$\psi(s):=\inf\{t\in\R:\, (s,t)\in A\cap\ell(s)\cap B(0,\delta)\},\, |s|<4r,$$
and proceed symmetrically. By the same reasoning, for each $x\in\graph\psi$ there exists a unit vector $v\in\Nor(A,x)$ with  
$\langle v,e_2\rangle\leq -\eta$, and, applying Lemma~\ref{L2_pom} for the set $A$ reflected by the $x_1$-axis, we obtain the Lipschitz property of $-\psi$. Clearly, $\psi\leq\varphi$, $\varphi(0)=\psi(0)=\varphi'(0)=\psi'(0)=0$ and
$$A\cap B(0,\delta)\cap V_r=\hypo\varphi\cap\epi\psi\cap B(0,\delta),$$ 
where $V_r:=\{(x_1,x_2):\, |x_1|<4r\}$.

Assume now that $|s|<r$ and  $x=(s,\varphi(s))$.
We know already that  there exists  $v\in \Nor(A,x)\cap S_{\R^2}$ with $\langle v,e_2\rangle\geq\eta$, hence,  since $r< r_0/2$, $B(x+rv,r)\cap A=\emptyset$ by Proposition~\ref{P_PR}~(vi).
 By the same argument as in the case (i) we obtain $B(x+rv,r)\subset B(0,\delta)\cap V_r$.
Since $(B(0,\delta)\cap V_r)\setminus\hypo\varphi$ is  clearly a component of $(B(0,\delta)\cap V_r)\setminus A$ and
 the ball $B(x+rv,r)\subset B(0,\delta)\cap V_r$ clearly intersects $(B(0,\delta)\cap V_r)\setminus\hypo\varphi$ since  $\langle v,e_2\rangle\geq \eta$,  we get $B(x+rv,r)\cap\hypo\varphi=\emptyset$. Thus we may apply \cite[Theorem~2.6]{Fu85} again and get that $\varphi$ is semiconcave on $(-r,r)$. By a symmetric argument one could verify the semiconvexity of $\psi$ on $(-r,r)$. Hence,
$A$ is of type $T^2$ at $a$. 

(iii) Finally, assume that $\Tan(A,a)$ is a ray. Applying a suitable isometry, we may assume that $a=0$ and $\Tan(A,0)=\{(s,0):\, s\geq 0\}$. 
Using Lemma~\ref{prilep}  we get that $\reach (A\cup[-\ep e_1,0],0)>0$ if $\ep>0$ is small enough. Clearly, $\Tan(A\cup[-\ep e_1,0],0)$ is the whole $x_1$ axis and we may apply the construction from (ii) and get Lipschitz functions $\psi\leq\varphi$ defined on an interval $(-r,r)$ such that $\varphi$ is semiconcave, $\psi$ semiconvex, $\varphi(0)=\psi(0)=\varphi'(0)=\psi'(0)=0$ and ($A\cup[-\ep e_1,0])\cap B(0,r)=(\hypo\varphi\cap\epi\psi)\cap B(0,r)$. Then, clearly, 
$A\cap B(0,r)=(\hypo\varphi|_{[0,r)}\cap\epi\psi|_{[0,r)})\cap B(0,r)$, thus, $A$ is of type $T^3$ at $a$.
\end{proof}
\begin{corollary}\label{kovro}
A compact set $\emptyset \neq A \subset \R^2$ has positive reach if and only if, for each 
 $a \in \partial A\setminus \is A$, $A$ is of type $T^i$ at $a$ for some    
 $i \in \{1,2,3\}$.  
\end{corollary}
\begin{remark}\label{kontt}
If $A \subset \R^2$ is a compact set with positive reach, then 
$$ \{a \in A:\ A \ \text{is of type}\ T^3\ \text{at}\ a\}\ \ \ \text{is finite}.$$
 Indeed, Definition \ref{D_types} and Corollary \ref{kovro} show that each point
 $x \in A$ has a neighbourhood containing at most one point at which $A$ is of type $T^3$.
\end{remark}

\begin{remark}\label{turn}
Let $A \subset \R^d$ be a connected compact set with positive reach. Lytchak  \cite[Theorems 1.2, 1.3] {Ly05}  proved that
 every different points $a_1 \in A$,  $a_2 \in A$ can be joined in $A$ by a simple $C^{1,1}$ curve.
We remind (Remark~\ref{LipPr}) that any two boundary points $b_1 \in \partial A$, $b_2 \in \partial A$ which belong to the same component of $\partial A$ can be joined by a rectifiable curve in $\partial A$ (but clearly not necessarily by  a simple $C^{1,1}$ curve). Theorem \ref{T_PR-char} easily implies that in the case $d=2$ such points $b_1$, $b_2$ can be joined in $\partial A$ by a more regular curve, e.g.\ by a curve with finite turn. (For the definition and a theory of curves with finite turn see \cite{Du} and the references therein). 
%Indeed, Theorem \ref{T_PR-char} and ????? easily imply that for each $x \in \partial A$ there exists $r_x>0$ such that $x$ can be joined with any $y \in \partial A \cap B(x, r_x)$ in $\partial A$ by a curve with finite turn. So our statement easily follows by a standard argument.
We do not know whether the statement holds for $d\geq 3$.	
\end{remark}

\section{Smooth points of sets with positive reach}\label{smooth}

To prove that a mapping $\vf: W \to V$ is $C^{1,1}$ (with controlled Lipschitz constant of $\vf'$),
we will use the following special version of
``Converse Taylor theorem''.
\begin{proposition}\label{contay} (\cite{Jo}, \cite{HJ}).
Let $W$, $V$ be finite-dimensional Hilbert spaces, $U=B(a,r)$  a ball in $W$ and $\vf: U \to V$
  a mapping. Suppose that there exists $c>0$ and for each $x \in U$ a linear
 mapping $g^x: W \to V$ such that
\begin{equation}\label{tayl}
|\vf(y) - (\vf(x) + g^x(y-x))| \leq c |y-x|^2\ \ \text{whenever}\ \ x,y \in U.
\end{equation}
Then $\vf\in C^{1,1}(U)$ and $\vf': U \to \mathcal L(W,V)$ is Lipschitz with constant
 $m c$, where $m>0$ is an absolute constant.
\end{proposition}
\begin{proof}
It is sufficient to use \cite[Chap. 1, Corollary 126]{HJ} (with $X:= V$, $Y:= W$, $k:=1$, $f:= \vf$ and
 $\omega(t):= ct$), observing that $e_U = 2$, since $U$ is a ball in $W$.
\end{proof}

\begin{proposition}\label{lipl}
Let $A \subset \R^d$ be a set with $\reach A > \rho >0$. Let $W \subset \R^d$ be a linear
 space of dimension $k$, where $1 \leq k \leq d-1$; denote $V:= W^{\perp}$. Let $K>0$,
 $U \subset W$ be an open ball in $W$ and $\vf:U \to V$ be a $K$-Lipschitz mapping such that
 $P:= \{w + \vf(w):\ w \in U\}$ is a relatively  open subset of $A$. Then $\vf$ is a $C^{1,1}$
 mapping and $\vf': U \to \mathcal L(W,V)$ is Lipschitz with constant $\mu (2+K)^3/\rho$,
 where $\mu$ is an absolute constant.
\end{proposition}
\begin{proof}
Let $D\subset U$ be the set of all points $w \in U$, for which there exists $\vf'(w)$.
It is well-known that $|\vf'(w)| \leq K$ for each $w \in D$. For $w \in D$, set $g^w:= \vf'(w)$.
Using the fact that $D$ is dense in $U$ (by Rademacher theorem) and compactness
 of $\{g\in \mathcal L(W,V):\ |g| \leq K\}$, we can easily to each $x \in U \setminus D$
 assign a linear mapping $g^x \in  \mathcal L(W,V)$ with $|g^x| \leq K$  and a sequence $(w^x_n) \subset D$
 such that $w^x_n \to x$ and $\vf'(w^x_n) = g^{w^x_n} \to g^x$. 
By Proposition \ref{contay} it is sufficent to verify that \eqref{tayl} holds
 with $c:= (2+K)^3/(2 \rho)$. To this end, consider arbitrary $x, y \in U$.

First consider {\it the case $x\in D$}. Set 
$$a:= x + \vf(x),\ \ \ L:= \{t+g^x(t):\ t\in W\}= \{t+\vf'(x)(t):\ t\in W\},\ \ \ M:=a +L.$$
Clearly, $M= \{s+ \vf(x)+ \vf'(x)(s-x):\ s \in W\}$.
It is a well-known and easy fact that $\Tan(P,a)= L$. Consequently, since $P$ is open in $A$,
 $\Tan(A,a)= L$. Set $z_1:= y + \vf(y)$. Using Proposition \ref{fedtan} (with $b:= z_1$, $t:= \rho$)
 we obtain
\begin{equation}\label{phve}
\dist(z_1, M)= \dist(z_1-a,L) \leq \frac{|z_1-a|^2}{2\rho}.
\end{equation}
Denote  $z_2:=y+ \vf(x)+ \vf'(x)(y-x) \in M$. 
Let $p \in M$ with $|p-z_1| = \dist(z_1, M)$ and  $w_p:=\pi_W(p)$.   Then
$$ |z_2-p| = |(y-w_p) + (g^x(y-w_p))| \leq (1 +K) |y-w_p|  \leq (1+K) |p-z_1|,$$
and consequently
\begin{equation}\label{fine}
 |z_2-z_1| \leq |z_2-p| + |p-z_1| \leq (2+K) |p-z_1|.
\end{equation}
The Lipschitzness of $\vf$ gives
\begin{equation}\label{thve}
|z_1-a| \leq |y-x| + |\vf(y)- \vf(x)| \leq  (1+K) |y-x|.
\end{equation}
Using  \eqref{fine}, \eqref{phve} and \eqref{thve}, we obtain
\begin{equation}\label{podt}
|\vf(y) - (\vf(x)+ g^x (y-x))| = |z_2-z_1| \leq (2+K)\cdot \frac{(1+K)^2 |y-x|^2}{2\rho} \leq c |y-x|^2,
\end{equation}
 and so \eqref{tayl} holds if $x \in D$.

In the {\it second case $x\in U\setminus D$}, we observe that by \eqref{podt}
\begin{equation}\label{coja}
|\vf(y) - (\vf(w^x_n)+ g^{w^x_n}(y-w^x_n))| \leq  c |y-w^x_n|^2,\ \ \ \text{for each}\ \ \ n.
\end{equation}
It is easy to see that $g^{w^x_n}(y-w^x_n) \to g^x(y-x)$ as $n\to \infty$, and so, passing to the
 limit in \eqref{coja}, we obtain the validity of \eqref{tayl} also in the second case.
\end{proof}
\begin{remark}\label{plb}
Proposition \ref{lipl} clearly implies (B) from Introduction.
\end{remark}

\begin{proposition}\label{Rjecjj}
  Let $A \subset \R^d$, $\reach A > \rho> 0$ and $d>k\geq 1$. 
	Suppose that $a \in A$, $\Tan(A,a)=:W$ is a $k$-dimensional vector space and
	$\dim(A \cap B(a,\delta))=k$ for some $\delta>0$. Denote $V:= W^{\perp}= \Nor(A,a)$.
	
	Then there exists a ball $U= B^W(c,r)$ in $W$ and a $C^{1,1}$ mapping $\vf: U \to V$
	 such that $a= c+\vf(c)$, $P:= \{w + \vf(w):\ w \in U\}$ is a relatively open subset of $A$ and
	$\vf'$ is  $(M/\rho)$-Lipschitz, where
	 $M$ is an absolute constant.
	 \end{proposition}
\begin{proof}
Set  $\omega:= \min\{\rho, \delta/2\}$ and $A^*:= A  \cap  \overline{B}(a,\omega)$.
 Then  $A^*$ has positive reach by Lemma \ref{L_souv}. We have $\dim A^*=k$, since clearly
  $\dim A^* \leq k$ and, by  \eqref{dimt},  $\dim A^* \geq \dim \Tan(A^*,a) = \dim \Tan(A,a) = k$.
	Since $(A^*)^{(k-1)}$ is closed in $A^*$ by Proposition \ref{zlfed} and $a \notin (A^*)^{(k-1)}$, there exists
	$0<\delta_0 < \omega$ such that $B(a,\delta_0) \cap A^* = B(a,\delta_0) \cap (A^* \setminus (A^*)^{(k-1)})$. So  \eqref{tanvec} gives that, if
	  $x\in A \cap B(a,\delta_0)$, then $\Tan(A,x)$ is a $k$-dimensional vector space and so $\Nor(A,x)$ is an 
 $(d-k)$-dimensional vector space. We can (and will) suppose that $a=0$. Since $V= \Nor(A,0)$, by
 Proposition \ref{P_PR}(iii) we can clearly find $0<\delta_1<\delta_0$ such that
\begin{equation}\label{vzd1}
 \text{$\dist(v, V \cap S_{\R^d})< 1/4$ if $v\in \Nor(A,x)\cap S_{\R^d}$ and $x \in B(0,\delta_1)\cap A$.}
\end{equation}
By \cite[Lemma 2]{HRW}, for each $x \in B(0,\delta_1)$  there exists a linear isometry $L: \R^d \to \R^d$ such that $L(V) = \Nor(A,x)$ and $L(\Nor(A,x)) = V$. Observing also that
 $L(S_{\R^d}) = S_{\R^d}$, it is easy to see that \eqref{vzd1} implies
\begin{equation}\label{symvzd}
 \text{$\dist(n, \Nor(A,x)\cap S_{\R^d})< 1/4$ if $n \in V\cap S_{\R^d}$ and $x \in B(0,\delta_1)\cap A$.}
\end{equation} 
Further choose $0< \delta_2 < \delta_1$ such that
\begin{equation}\label{deldv}
 \delta_2 < \frac{\rho}{8}.
\end{equation}
Now consider two arbitrary points  $x_1$, $x_2$ in $A \cap B(0,\delta_2)$ and write
 $x_1= w_1 + v_1$, $x_2= w_2 + v_2$, where  $w_1, w_2 \in W$ and $v_1, v_2 \in V$.
 We will show that
\begin{equation}\label{jelip}
|v_1-v_2| \leq |w_1-w_2|.
\end{equation}
 So suppose, to the contrary, that $|v_1-v_2| > |w_1-w_2|$. Then clearly $|v_1-v_2| \geq (1/2) |x_2-x_1|$.
Applying \eqref{symvzd} for $n:= (v_2-v_1)/|v_2-v_1|$, we can choose $n_1 \in \Nor(x_1,A)
 \cap S_{\R^d}$  with  $|n-n_1| < 1/4$. Using Proposition \ref{P_PR}(iv) and \eqref{deldv},
 we obtain
\begin{equation}\label{prner}
\langle x_2-x_1, n_1 \rangle  \leq \frac{|x_2-x_1|^2}{2 \rho} \leq \frac{1}{8} |x_2-x_1|.
\end{equation}
On the other hand, using  $|n-n_1| < 1/4$ and $|v_1-v_2| \geq (1/2) |x_2-x_1|$, we obtain
\begin{multline*}
\langle x_2-x_1, n_1 \rangle = \langle x_2-x_1, n \rangle + \langle x_2-x_1, n_1-n \rangle  \\
   =|v_2-v_1| + \langle x_2-x_1, n_1-n \rangle  \geq \frac{1}{2} |x_2-x_1| - \frac{1}{4} |x_2-x_1|
	 = \frac{1}{4} |x_2-x_1|,
\end{multline*}
which contradicts  \eqref{prner}.
So \eqref{jelip} holds. Therefore there exists a set $D \subset W$ and a $1$-Lipschitz mapping $\psi:  D \to V$ such that 
$$ A \cap B(0, \delta_2) = \{w+ \psi(w):\ w \in D\}.$$
Now we will show that $W \cap B(0, \delta_2/8) \subset D$, i.e.,
\begin{equation}\label{eprj}
\forall w \in B(0,  \delta_2/8)\cap W\ \exists  v \in V:\ w+v \in A \cap B(0, \delta_2).
\end{equation}
To this end, fix an arbitrary $w \in B(0, \frac{1}{8} \delta_2)\cap W$, denote
 $S(w):= w + V$ and suppose, to the contrary, that  $S(w) \cap B(0, \delta_2) \cap A = \emptyset$.
 Then also  $S(w) \cap \overline{B}(0, \delta_2/2) \cap A = \emptyset$ and so we can find
 $c \in S(w)$ and $d \in \overline{B}(0, \delta_2/2) \cap A$ such that
$$ |d-c| = \dist(S(w), \overline{B}(0, \delta_2/2) \cap A)  >0.$$
Since  $0 \in A$ and $|w| < \delta_2/8$, we have  $|c-d| < \delta_2/8$. 
 Obviously  $\Pi_{S(w)}(d) =c$, and thus $d-c \in W$. Writting $d= w^* + v^*$, where
 $w^* \in W$ and $v^* \in V$, we obtain
$$ |w^*| \leq |w| + |w^*-w|= |w| + |d-c| < \frac{1}{8} \delta_2 + \frac{1}{8} \delta_2 = \frac{1}{4} \delta_2. $$
Since by \eqref{jelip}  $|v^*| \leq |w^*|$, we obtain that $|d| <  \delta_2/2$, i.e.
 $d \in B(0, \delta_2/2)$. Thus, for all sufficiently small $0<t < |d-c|$, 
 $$B(d + t (c-d), t |c-d|) \subset B(0, \delta_2/2),$$ 
which clearly implies
  $\Pi_{A}(d + t(c-d)) = d$ and so $n^*:= (c-d)/|c-d| \in \Nor(A,d)$ by Proposition \ref{P_PR}(vi).
	 Since we know that $n^* \in W:= V^{\perp}$, we clearly obtain a contradiction with
 \eqref{vzd1}.

Applying Proposition \ref{lipl} with  $U:= B(0, \delta_2/8)\cap W$, $\vf:= \psi\restriction_U$ and $K=1$, we easily obtain our assertion.
\end{proof}
Now we will prove our main theorem on general sets of positive reach in any dimension, which
 contains Federer's result \eqref{fecjj}.

\begin{theorem}\label{hlav}
Let $A \subset \R^d$ be a set of positive reach with $0<k:= \dim A < d$. Then
\begin{enumerate}
\item[(i)]  
$A$ can be locally covered by finitely many  DC surfaces of dimension $k$.
\item[(ii)]  
$R:= A\setminus A^{(k-1)}\neq \emptyset$ is a uniform $C^{1,1}$ manifold of dimension $k$ which is open in $A$ and 
$A\setminus R= A^{(k-1)}$  can be locally covered by finitely many  DC surfaces of dimension $k-1$.
\end{enumerate}
\end{theorem}
\begin{proof}
To prove (i), observe that for each $a\in A$ (by  \eqref{dimt}) $\dim \Tan(A,a) \leq k$ and therefore $\Nor(A,a)$ contains a vector space of dimension $d-k$. Consequently clearly
 $A \subset A_1^{(k)}$, and so (i) follows from Proposition \ref{poakep}.

The first part of (ii) follows from Proposition \ref{Rjecjj}, since $\Tan(A,a)$ is a $k$-dimen\-sio\-nal
 vector space for each $a\in  R:= A\setminus A^{(k-1)}$ by $\eqref{tanvec}$. The second part
 of (ii) was proved in Proposition \ref{zlfed}. 
\end{proof}

\begin{remark}\label{kehla}(To Theorem \ref{hlav}.)
\begin{enumerate}
\item [(a)]
If $A$ is compact, we can clearly omit (both in (i) and in (ii)) ``locally''.
\item[(b)]
Writing ``manifolds'' instead of ``surfaces'' we can  omit (both in (i) and in (ii)) ``locally''
 also in the case of a non-compact $A$. This follows rather easily from
  the facts that $N=N(d,k)$ in    Proposition \ref{zlfed}
 and $N=N(d,k, \ep)$ in Proposition \ref{poakep} (which is applied, in the proof of (i), with $\ep=1$).
\item[(c)]  
It can be shown that $R$ cannot be always locally covered by  finitely many  $C^{1,1}$ surfaces of dimension $k$ (see Example~\ref{Ex2}).
\item[(d)] If $j<k-1$, the set $A^{(j)}$ cannot be always covered by a {\it locally finite} system of DC surfaces of dimension $j$. As an example, consider a convex body $A\subset\R^2$ with $A^{(0)}\subset \partial A$ infinite.
\end{enumerate}
\end{remark}

Further we will consider also relatively open subsets $\emptyset \neq B \subset A$, for which
 $\dim B \neq \dim A$. We will need the following notation.

\begin{definition}\label{dksk}  \rm
Let $A \subset \R^d$ be a set of positive reach and let $0\leq k \leq d$. We denote 
\begin{enumerate}
\item[(a)]
by $D_k(A)$ the set of points  $a\in A$, such that $\dim(A\cap B(a,r)) = k$ for all sufficiently
 small $r>0$, and
\item[(b)]
by $S_k(A)$ the set of points $a\in A$, such that $A\cap B(a,r)$ is a $k$-dimensional
 $C^{1,1}$ manifold for some $r>0$.
\end{enumerate}
\end{definition}
%$T_k$.......the set of points $a\in A$, where $Tan(A,a)$ is a $k$-dimensional vector space.
\begin{remark}
\begin{enumerate}
\item[(i)]
It is easy to see that $A = \bigcup_{k=0}^{d} D_k(A)  =  \is A \cup \bigcup_{k=1}^{d} D_k(A)$.  
\item[(ii)]
Clearly each $S_k(A)$ is open in $A$.
\item[(iii)]
If $\dim A= k \geq 1$, then 
\begin{equation}\label{ssk}
R:= A^{(k)} \setminus A^{(k-1)} = S_k(A).
\end{equation}
Indeed, $R \subset S_k(A)$ follows from Theorem \ref{hlav} and $S_k(A) \subset A^{(k)} \setminus A^{(k-1)}$
 follows from the obvious fact that $\Tan(A,x)$ is a $k$-dimensional vector space for each $x\in S_k(A)$.
\end{enumerate}
\end{remark}
%Remark:  above we can write equivalently  ``$C^{1}$ manifold'' or ``Lipschitz manifold'' (it follows
% from ....) or even ``topological manifold'' (it follows from a result stated by Federer and
% proved by Lytchak).
\begin{proposition}\label{P_dksk}
Let $A \subset \R^d$ be a set of positive reach and let $1\leq k \leq d$. Then
%\begin{enumerate}
%\item [(i)]
$S_k(A)$ is a dense subset of $D_k(A)$.
%\item[(ii)]
%$D_k(A) \setminus S_k(A)$ can be locally covered by finitely many of DC surfaces of dimension $k-1$.
%\end{enumerate}
\end{proposition}
\begin{proof}
Suppose $D_k(A) \neq \emptyset$ and choose arbitrary $d\in D_k(A)$ and $\delta >0$.
 Since $d\in D_k(A)$, we can choose $0<\omega < \delta$ such that $\omega < \reach(A)$
 and $\dim (A \cap B(a,r))= k$ for all $0< r \leq \omega$.
By  Lemma \ref{L_souv}, $A^*:= A \cap \overline B(a,\omega)$ has positive reach.
 As (e.g., by Proposition \ref{sipr}) $\dim((A^*)^{k-1} \cap B(a, \omega/2)) \leq k-1$ and
 $\dim (A^* \cap B(a,\omega/2))= k$, we have $(A^* \setminus (A^*)^{k-1}) \cap B(a, \omega/2) \neq
 \emptyset$. Since $A^* \setminus (A^*)^{k-1}$ is by Theorem \ref{hlav} a $k$-dimensional $C^{1,1}$ manifold
 open in $A^*$, we easily obtain  $S_k(A) \cap B(a,\omega/2) \neq \emptyset$ and (i) follows.
  \end{proof}
\begin{corollary}\label{huhl}
The set of all smooth points  $S:= S_0 \cup S_1\cup\dots \cup S_d$ is open and dense
in $A$.
\end{corollary}
However, the structure of the set $A\setminus S$, which is nowhere dense in $A$, can be very complicated and a satisfactory complete characterization even of the local structure of sets of positive reach
 in $\R^d$ for $d\geq 3$ seems to be a very difficult task.

In such spaces we are not able even answer the following natural question:

\begin{question}
Let $A \subset \R^d$ be a compact set with positive reach. Does there exists a decomposition
\begin{equation}\label{strat} A = Q_1 \cup \dots \cup Q_s,
\end{equation}
where $Q_i$, $i=1,\dots,s$, are pairwise disjoint and each $Q_i$ is a relatively open subset of a DC surface $P_i$ of dimension $0 \leq k_i \leq d$?
\end{question}

\begin{remark}
\begin{enumerate}
\item[(i)] Theorem \ref{T_PR-char}  implies that for $d=2$ the answer to Question is positive.
Indeed, first observe that if each of closed sets $A_1,\dots, A_k$ has a decomposition of type
 \eqref{strat}, then also their union $ A_1\cup\dots\cup A_k = A_1 \cup (A_2\setminus A_1) \cup \dots$
 has clearly such a decomposition.

Further, observe that Theorem \ref{T_PR-char}  easily implies that for each $x \in \partial A$ there exists
 an open neighbourhood $U_x$ of $x$ (e.g. a suitable open square) such that $\partial A \cup \overline{U_x}$ has a decomposition of type \eqref{strat}. Using compactness of $\partial A$
 and the first observation, we easily conclude the proof. 
\item[(ii)] If $d\geq 3$, we
 do not even know whether there always exists a weaker decomposition, in which $Q_i$ is supposed to
 be a relatively open subset of a DC manifold.
\item[(iii)] 
 If such a decomposition exists, it is not (clearly) uniquely determined. Moreover, already in the plane in some cases there is no ``canonical decomposition'' (see Example~\ref{Ex1}).
\end{enumerate}
\end{remark}

\begin{example}  \label{Ex1}
Let $\emptyset\neq K\subset\R$ be compact and denote $I:=\conv K$,
$$\varphi: x\mapsto (\dist (x,K))^2,\quad x\in \R,$$
and 
$$A_K:=\hypo\varphi \cap \epi(-\varphi) \cap (I \times \R).$$
The function $\varphi$ is semiconcave (with semiconcavity constant $2$), see \cite[Proposition 2.2.2]{CS}, and, hence, $\reach A_K>0$ by Corollary~\ref{kovro}. 

In what follows, we will identify $\R$ with the $x$-axis $\R \times \{0\}$.
The following properties can be easily   shown.
\begin{enumerate}
\item $A_K$ is topologically regular ($A_K=\operatorname{cl}(\INt A_K)$) if $K$ is totally disconnected, nevertheless, the boundary $\partial A_K$ fails to be a $1$-dimensional manifold at all points of $K$ (note that $K$ can even have positive one-dimensional measure).
\item Both $A_K$ and $\partial A_K$ are (arcwise) connected. However, if $K$ is infinite and totally disconnected, the interior of $A_K$ has infinitely many components and the boundary $\partial A_K$ is not locally contractible at accumulation points of $K$. 
\item $\partial A_K\setminus\partial I$ can be written as the union of two connected one-dimensional DC manifolds which can be chosen in an infinite number of ways. Moreover, there is no ``canonical'' way how to decompose $\partial A_K\setminus\partial I$ into two disjoint one-dimensional DC manifolds (we could choose, e.g., $M_1=(\partial A_K\setminus\partial I)\cap\{y\geq 0\}$ and $M_2=\partial A_K\cap\{y< 0\}$).
\end{enumerate}
\end{example}

\begin{example}  \label{Ex2}
Consider the last example with $K$ infinite and totally disconnected to be embedded into the $x,y$ plane in $\R^3$. Note that we can write
$$A_K=\operatorname{cl}\left(\bigcup_{i=1}^\infty A_{\{a_i,b_i\}}\right),$$
if $(a_i,b_i)$ are all the maximal open intervals in $I\setminus K$, $i=1,2,\ldots$, and that the intersection $A_{\{a_i,b_i\}}\cap A_{\{a_j,b_j\}}$ is equal to the intersection of the segments $[a_i,b_i]\cap [a_j,b_j]$ if $i\neq j$. Let $T_\theta$ denote the rotation in $\R^3$ around the $x$-axis by an angle $\theta\in [0,2\pi)$ in the positive sense. It is not difficult to show that, choosing any sequence $\Theta=(\theta_i)_{i=1}^\infty$, the set
$$A_K^\Theta:=\operatorname{cl}\left(\bigcup_{i=1}^\infty T_{\theta_i}(A_{\{a_i,b_i\}})\right)$$
has positive reach and $\dim A_K^\Theta=2$.
Note that the ``regular part'' of $A^\Theta_K$ (cf.\ Theorem~\ref{hlav}) is
$$R:=A^\Theta_K\setminus(A^\Theta_K)^{(1)}=
(A^\Theta_K)^{(2)}\setminus(A^\Theta_K)^{(1)}=\bigcup_{i=1}^\infty  T_{\theta_i}(\INt A_{\{a_i,b_i\}}).$$  
We observe the following.
\begin{enumerate}
\item $R$ is a uniformly $C^{1,1}$ $2$-dimensional manifold (cf.\ Theorem~\ref{hlav}). Nevertheless, choosing an appropriate sequence $\Theta$, the function 
$$x\mapsto\Tan(R,x)$$ 
is not globally Lipschitz on $R$. 
\item Let  $x$ be an accumulation point of $K$  and choose a sequence of angles $(\theta_i)$ 
so that, for any $\delta>0$, the set of those $\theta_i$ which correspond to intervals 
 $(a_i,b_i) \subset (x-\delta, x+ \delta)$ is dense in $[0,2 \pi)$. 
 Then it is easy to see that the regular part $R$ cannot be covered by finitely many $C^{1,1}$ hypersurfaces. (However, $R$, and even $A^\Theta_K$, can be covered by finitely many  DC
 hypersurfaces by Remark \ref{kehla} (a).)  
\end{enumerate}
\end{example}

\section{One-dimensional PR sets}\label{oned}

\begin{definition}[Intrinsic distance] \rm
For $A\subset\R^d$ and $x,y\in A$ we set
$$d_A(x,y):=\inf\{\length(\gamma);\, \gamma:[a,b]\to A \text{ continuous}, \gamma(a)=x,\gamma(b)=y\}$$
and call it {\it intrinsic distance} of $x,y$ in $A$. 
\end{definition}

\begin{remark}
Note that the intrinsic distance may take infinite values, hence, it is not a distance in the standard metric spaces setting. Nevertheless, it clearly has all other  properties of a metric (cf.\ \cite[Chapter~1]{BBY01}).
\end{remark}

\begin{definition} \label{Def_manif}   \rm
Let $k\in\{1,\dots,d\}$ be given.
A set $\emptyset \neq M\subset\R^d$ is a {\it $k$-dimensional $C^1$ submanifold with boundary} if for any $x\in M$ there exist a neighbourhood $U$ of $x$ and a $C^1$ diffeomorphism $\phi: M\cap U\to V$ of $M\cap U$ onto a relatively open subset $V$ of a closed halfspace in $\R^k$ (i.e., $\phi$ is a homeomorphism and both $\phi$ and $\phi^{-1}$ are restrictions of $C^1$ mappings defined on some open sets). The mapping $\phi$ is called a {\it local $C^1$ chart} of $M$. 
\end{definition}

\begin{definition}\label{cjjpar}  \rm
A {\it $C^{1,1}$ curve} (with parameter $L$) is the image of some $C^1$ arc-length parametrization $\gamma: I\to\R^d$ defined on a (nondegenerate) interval and such that the derivative $\gamma'$ is $L$-Lipschitz. We call then $\gamma$ a $C^{1,1}$ {\it arc-length parametrization} (with parameter $L$). A $C^{1,1}$ curve is 
\begin{enumerate}
\item {\it simple} if it has a $C^{1,1}$ arc-length parametrization which is a homeomorphism, 
\item {\it closed simple} if it has a $C^{1,1}$ arc-length parametrization $\gamma: [a,b]\to\R^d$ such that $\gamma(a)=\gamma(b)$, $\gamma'(a)=\gamma'(b)$ and $\gamma|_{[a,b)}$ is injective,
\item a {\it simple $C^{1,1}$ arc} if it is a compact and simple $C^{1,1}$ curve.
\end{enumerate}
A $C^1$ curve (simple $C^1$ curve, closed simple $C^1$ curve, simple $C^1$ arc) is defined as above, but without the Lipschitz property of the derivative.
\end{definition}
\begin{remark}\label{regces}
Let $I \subset \R$ be an interval, $\psi: I \to \R^d$ a $C^1$ homeomorphism,
 $\lambda:= \inf \{|\psi'(x)|: x \in I\} > 0$, and let $\psi'$ be $M$-Lipschitz.  Then $\psi(I)$
 is a simple $C^{1,1}$ curve with parameter $2M/\lambda^2$.

Indeed, a standard computation (see., e.g.,   \cite[Lemma~2.7]{DZ2}) reveals   
 that any arc-length reparametrization of $\psi$ has  $(2M/\lambda^2)$-Lipschitz derivative.
\end{remark}

\begin{remark}\label{47}
 Let $\gamma: I\to\R^d$ be an arbitrary (not necessary injective) arc-length parametrization
  $C^{1,1}$ with parameter $L>0$ and   $x=\gamma(s)$, $y=\gamma(t)$ ($s,t \in I$).
	The mean value theorem for vector valued functions implies (see \cite[(8.6.2)]{D})
$$|y-x-\gamma'(s)(t-s)|\leq L|t-s|^2.$$
It follows that if $|t-s|<1/(2L)$ then (note that $|\gamma'(s)|=1$)
\begin{equation} \label{E_cc}
|y-x|\geq |\gamma'(s)||t-s|-L|t-s|^2\geq \frac 12 |t-s|,
\end{equation}
hence,
\begin{equation}  \label{E_tan_c}
|y-x-\gamma'(s)(t-s)|\leq 4L|y-x|^2.
\end{equation}
	\end{remark}

\begin{lemma} \label{L_curve}
A connected one-dimensional $C^1$ submanifold of $\R^d$ with boundary is a simple or closed simple $C^1$ curve.
\end{lemma}

\begin{proof}
Milnor \cite[Appendix]{Mil97} showed that a connected one-dimensional $C^\infty$ submanifold  of $\R^d$ is $C^\infty$ diffeomorphic to a circle or interval. One can easily check that his proof works for a connected one-dimensional $C^1$ submanifold $M \subset \R^d$ as well, yielding that it is $C^1$ diffeomorphic to a circle or to an interval. 
Now a standard straightforward argument gives our assertion.
\end{proof}

\begin{theorem}   \label{T_PR_1}
Let $A\subset\R^d$ be a closed set with $\dim A=1$. Then, $\reach A>0$ if and only if the following two statements hold:
\begin{enumerate}
\item[(i)] there exists an $L>0$ such that each component of $A$ is a singleton or a simple or closed simple $C^{1,1}$ curve with parameter $L$;
\item[(ii)] the inverse of the embedding $A\hookrightarrow\R^d$ is uniformly continuous (with respect to the intrinsic distance in $A$), i.e., for any $\ep>0$ there exists $\delta>0$ such that 
$$|x-y|<\delta \implies d_A(x,y)<\ep,\quad x,y\in A.$$
\end{enumerate}
\end{theorem}

\begin{proof}
Assume first that (i) and (ii) hold and let $L>0$ be a parameter from (i). We shall show that there exists $\rho>0$ such that
\begin{equation}  \label{E_reach}
\dist(y-x,\Tan(A,x))\leq\frac{|y-x|^2}{2\rho},\quad x,y\in A,
\end{equation}
which is equivalent to $\reach A\geq\rho$ (see Proposition~\ref{fedtan}).

Choose a $\delta>0$ which corresponds by (ii) to $\ep:= 1/(2L)$ and set 
 $\rho:=\min\{\delta/2,(1/(8L)\}$. To prove \eqref{E_reach}, 
consider two arbitrary different points $x,y\in A$. 

If  $|y-x|\geq\delta$ then we get
$\dist(y-x,\Tan(A,x))\leq |y-x|\leq|y-x|^2/\delta$, and \eqref{E_reach} follows  since $\delta\geq 2\rho$.

If $|y-x|<\delta$ then  $d_A(x,y)<1/(2L)$       by the choice of $\delta$.
Consequently $x$ and $y$ belong to the same component $C$ of $A$. Using (i), it is easy to show
 that we  can choose an arc-length $C^{1,1}$ parametrization $\gamma: I \to \R^d$ of $C$ with parameter $L$ and points $s,t \in I$ such that $x=\gamma(s)$, $y= \gamma(t)$ and $|t-s|=d_A(x,y)$ 
(this is not quite obvious only if
 $C$ is a simple closed curve). Then  $|t-s|<1/(2L)$ and so  \eqref{E_tan_c} holds. Since
 $4L \leq 1/(2\rho)$ and  $\gamma'(s)(t-s)\in\Tan(A,x)$, we see that  \eqref{E_tan_c} implies
  \eqref{E_reach}.

Now we prove the second implication. Assume that $\reach A>0$, fix some $0<\rho<\reach A$, let $x\in R$ be a point of the regular part of $A$, $R=A \setminus A^{(0)}$, and denote $W:=\Tan(A,x)$ (this is a one-dimensional space by \eqref{tanvec}). Proposition~\ref{Rjecjj} 
yields that there exist an  open neighbourhood $U$ of $x$, $c\in W$, $r>0$ and $C^1$ mapping 
$$\varphi: B(c,r)\cap W\to W^\perp$$ 
with $(m/\rho)$-Lipschitz derivative  (where $m$ is  an absolute constant)   such that $c+\varphi(c)=x$ and, denoting $V:=B(c,r)\cap W$, 
$$\{t+\varphi(t):\, t\in V\}=A\cap U.$$
Without any loss of generality we will suppose that  $W= \spa \{e_1\}$ and  identify in the usual way      
 $W$ with $\R$ and $W^{\perp}$ with $\R^{d-1}$. 
Now it is easy to see that $A\cap U$ is a simple $C^{1}$ curve with parametrization
$$\psi: t\mapsto (t,\varphi(t)),\quad t\in V= (c-r,c+r).$$
Moreover, $\phi:=\psi^{-1}$ is clearly a local $C^1$ chart of $A$ (in the sense of Definition~\ref{Def_manif}) and $|\psi'|\leq 1$.

Consider now a point $x\in A^{(0)}\setminus\is A$. Then, $\Tan(A,x)$ is a ray determined by a unit vector $u$, and Lemma~\ref{prilep} shows that $\reach(A\cup[x,x-\frac{\rho}4 u])\geq\rho/4$. Clearly, $x$ belongs to the regular part of $A\cup[x,x-\frac{\rho}4 u]$ and 
the above consideration shows that there exists an open neighbourhood $U$ of $x$ and a $C^1$ diffeomorphism $\psi$ of an open interval $V$ onto $(A\cup[x,x-\frac{\rho}4 u])\cap U$ such that $|\psi'|\geq 1$ and $\psi'$ is $(4m/\rho)$-Lipschitz. Moreover, $J:=\psi^{-1}(A)$ is a half-open interval, and $\psi\restriction_J$ is a $C^1$ diffeomorphism of $J$ onto $A\cap U$, hence, its inverse is a local $C^1$ chart of $A$. 

Summarizing both cases, to any point $x\in A\setminus\is A$ there exists a neighbourhood $U_x$, open or half-open interval $J_x$ and $C^1$ diffeomorphism $\psi_x:J_x\to A\cap U_x$ such that $\psi_x'$ is $(4m/\rho)$-Lipschitz and $|\psi_x'|\geq 1$. In
 particular, $\psi_x^{-1}$ is a local $C^1$ chart of $A$, hence, $A\setminus\is A$ is a one-dimensional $C^1$ submanifold with boundary. By Lemma~\ref{L_curve}, any connected component $C$ of $A$ which has more than one point must be a $C^1$ simple or closed simple curve. Let
 $\gamma: I \to \R^d$ be a corresponding $C^1$ arc-length parametrization of $C$ (injective on $\INt I$).
 If $x\in C$, choose $U_x$, $J_x$ and $\psi_x$ as above. Since $\psi_x$ parametrizes some relative
 neighbourhhood of $x$ in $C$, using Remark \ref{regces}  we obtain that $\gamma$
  has locally $L$-Lipschitz derivative with $L:={8}m/\rho$.
 This clearly implies that $\gamma'$ is globally $L$-Lipschitz and so condition (i) holds.

It remains to verify condition (ii). Set $\delta:=\min\{\rho/2,1/(4L)\}$. Take two different points $x,y\in A$ at distance $|y-x|<\delta$. Using Lemma~\ref{L_souv}, we get that $A\cap B_{x,y}$ is connected, where $B_{x,y}$ is the closed ball of diameter $|y-x|$ containing $x$ and $y$. Thus, $x$ and $y$ lie in one connected component $C$ of $A$, which is (by already proved condition (i))
 a simple or closed simple $C^{1,1}$ curve with parameter $L$. 
Since $A\cap B_{x,y}$ is connected, we have that $A\cap B_{x,y}= C \cap B_{x,y} $ 
  and we easily see that there exists a $\gamma:[s,t]\to C \cap B_{x,y}$ which is a simple  $C^{1,1}$  arc-length parametrization
	 with parameter $L$ (of a simple subarc of $C$) such that $\gamma(s)=x$ and $\gamma(t)=y$.
If $s_1:=s+1/(2L)\in [s,t]$ then, using \eqref{E_cc}, we get
$$|\gamma(s_1)-x|\geq\frac 12|s_1-s|=\frac 1{4L}>|y-x|,$$ 
hence, $\gamma(s_1)\not\in B_{x,y}$. Consequently, $|t-s|\leq 1/(2L)$ and, using \eqref{E_cc} again, we get
$d_A(x,y)\leq |t-s|\leq 2|y-x|$. This clearly proves (ii) and the proof is complete.
\end{proof}

We say that a simple $C^{1,1}$ curve $A\subset\R^d$ has the {\it quasi-arc property}, provided that
\begin{enumerate}
\item[(Q)]
For each $\ep>0$ there exists $\delta>0$ such that $|x_1-x_2|< \ep$ whenever 
 $\{x_1,x_2,x_3\} \subset A$, $|x_1-x_3|< \delta$ and $x_1$ and $x_3$ belong to different components of $A \setminus \{x_2\}$.
\end{enumerate}

\begin{corollary}\label{osse}
Let $A\subset\R^d$ be a closed connected set with $\dim A=1$. Then, $\reach A>0$ if and only if it is a $C^{1,1}$ curve of one of the following four types:
\begin{enumerate}
\item[(a)] $A$ is a simple $C^{1,1}$ arc,
\item[(b)] $A$ is a closed simple $C^{1,1}$ curve,
\item[(c)] $A$ is a simple $C^{1,1}$ curve homeomorphic to $[0,\infty)$ and with the quasi-arc property,
\item[(d)] $A$ is a simple $C^{1,1}$ curve homeomorphic to $\R$ and with the quasi-arc property.
\end{enumerate}
\end{corollary}

\begin{proof}
If $A$ has positive reach then, by Theorem~\ref{T_PR_1}, it is a simple or closed simple $C^{1,1}$ curve. If (a) or (b) holds, we are done. In the opposite case we can choose a
 $C^{1,1}$ arc-length parametrization $\gamma:I\to A$, which is a homeomorphism between $I$ and $A$
 and $I$ is not compact (since (a) does not hold).  
Now observe that
 $I$ is a closed set. Indeed, assume for the contrary that there exists a point $t\in\overline{I}\setminus I$, and let $t_i\in I$ be such that $t_i\to t$, $i\to \infty$. Since $\gamma$ is $1$-Lipschitz, there exists $\lim \gamma(t_i)=:x \in \R^d$, and we get $x\in A$ from the closedness of $A$. Since $\gamma^{-1}$ is continuous, we get $t_i\to\gamma^{-1}(x)\in I$, which is a contradiction. 
So, since $I$ is not compact, $I$ is either isometric with $[0,\infty)$, or equal to $(-\infty,\infty)$. We shall verify property (Q). Let $\ep>0$ be given, set $\delta:=\min\{\frac\ep 2,\frac 12\reach A\}$ and let $x_1,x_2,x_3\in A$ be such that $|x_1-x_3|<\delta$ and $x_1,x_3$ belong to different components of $A\setminus\{x_2\}$. Using Lemma~\ref{L_souv}, we get that $A\cap\overline B(x_1,\delta)$ is connected. Hence, $|x_2-x_1|\leq\delta<\ep$, which proves (Q).

To prove the other implication, suppose that $A$ is of a type (a)-(d). Then (i) from
 Theorem~\ref{T_PR_1} trivially 
 holds and so, by Theorem~\ref{T_PR_1}, it is sufficient to verify property (ii) from this theorem. 
If $A$ is a curve of type (a) or (b), it must satisfy property (ii) (indeed, it is easy to show that in these cases, the inverse to the embedding $A\hookrightarrow\R^d$ is continuous, and the uniform continuity follows from the compactness of $A$). 

So suppose that $A$ is of type (c) or (d), we have to verify condition (ii).
To this end, let $\gamma: I \to A$ be an arc-length $C^{1,1}$
 parametrization of $A$ with parameter $L$ which is a homeomorphism between $I$ and $A$.
 Note that clearly   $|t_1-t_2|= d_A(\gamma(t_1), \gamma(t_2))$ whenever $t_1, t_2 \in I$.

 Given $\ep_0:= 1/(8L)$, find $\delta_0>0$ by condition (Q).
Now, for arbitrary $\ep>0$, put  $\delta:=\min\{\ep/2, \delta_0\}$. It is sufficient to prove
 that $|t_1-t_2|< \ep$ whenever  $0\leq t_1<t_2$ and $|\gamma(t_1)-\gamma(t_2)| < \delta$.
 Suppose, to the contrary, that $0\leq t_1<t_2$, $|\gamma(t_1)-\gamma(t_2)| < \delta$ and
$|t_1-t_2|\geq \ep$. Then we have  $|t_1-t_2| \geq 1/(2L)$, since otherwise by \eqref{E_cc}
 $|\gamma(t_1)-\gamma(t_2)|  \geq |t_1-t_2|/2 \geq \ep/2  \geq \delta$.
 So, setting  $x_2:= \gamma(t_1 + 1/(4L))$, we clearly have that $x_1:=\gamma(t_1)$ and $x_3:=\gamma(t_2)$ belong to
 different components of $A\setminus \{x_2\}$ and so $|x_1-x_2| < \ep_0= 1/(8L)$ by the choice
 of $\delta_0$. On the other hand, by \eqref{E_cc} we obtain $|x_1-x_2| \geq (1/2) (1/(4L))=
 1/(8L)$, a contradiction.
\end{proof}

\begin{remark}
We have shown that in case (c), $A$ admits a homeomorphic arc-length parametrization $\gamma:[0,\infty)\to A$. Moreover, we have $\lim_{t\to\infty}|\gamma(t)|=\infty$. Indeed, otherwise there exists a sequence $t_i \to \infty$ such that $\gamma(t_i) \to x \in \R^d$. Since $A$ is closed, we have $x \in A$.
 Using the continuity of $\gamma^{-1}$, we get $t_i\to\gamma^{-1}(x)\in I$, which is a contradiction.

Similarly, $\lim_{t\to\pm\infty}|\gamma(t)|=\infty$ in case (d).
\end{remark}

\begin{example}\label{smycka}
Set $\psi(t):= (t^2, t e^{-t^2}),\ t \in \R$. Using Remark \ref{regces},
 it is easy to check that the image of $\psi$ is a simple $C^{1,1}$
 curve (and it is also a one-dimensional $C^{1,1}$ manifold). However, it has neither the quasi-arc property, nor positive reach.         
\end{example}


\begin{thebibliography}{99}
\bibitem{Al}
G. Alberti, {\it On the structure of singular sets of convex functions},
Calc. Var. Partial Differential Equations {\bf 2} (1994),  17--27. 

\bibitem{AAC}
G. Alberti, L. Ambrosio and P. Cannarsa,
{\it On the singularities of convex functions},
Manuscripta Math. {\bf 76} (1992),  421--435.

\bibitem{Ar}
 N. Aronszajn, {\it Differentiability of Lipschitzian mappings between Banach spaces}, Studia Math. 
{\bf 57} (1976), 147--190. 

\bibitem{Ba82} 
V. Bangert, {\it Sets with positive reach}, Arch.\ Math. (Basel) {\bf 38} (1982), 54--57.

\bibitem{BBY01} 
D. Burago, Y. Burago and S. Ivanov, {\it A Course in Metric Geometry}, Amer.\ Math.\ Soc., Providence, 2001.

\bibitem{Cl90} 
F. Clarke, {\it Optimization and Nonsmooth Analysis}, SIAM, Philadelphia, 1990.

\bibitem{CP}
A. Colesanti and C. Pucci, 
{\it Qualitative and quantitative results for sets of singular points of convex bodies}, 
Forum Math. 9 (1997),  103--125. 

\bibitem{CT10} 
G. Colombo and L. Thibault, {\it Prox-regular sets and applications}, In: Handbook of Nonconvex Analysis and Applications, D.Y. Gao and D. Motreanu Eds., International Press, Boston, 2010, pp.~99--182.

\bibitem{CF} 
G. Crasta and  I. Fragala, {\it On the characterization of some classes of proximally smooth sets}, ESAIM: Control, Optimisation and Calculus of Variations (to appear), DOI: 10.1051/cocv/2015022.

\bibitem{CH}
A. Colesanti and D. Hug, {\it Hessian measures of semi-convex functions and applications to support measures of convex bodies}, Manuscripta Math. {\bf 101} (2000),  209–-238.
 
\bibitem{CS}
P.~Cannarsa and C.~Sinestrari,
{\it Semiconcave Functions, Hamilton-Jacobi Equations, and Optimal Control}, 
Progress in Nonlinear Differential Equations and their Applications 58, Birkh\"auser, Boston, 2004.

\bibitem{Da}
J.  Dalphin, {\it  Some characterizations of a uniform ball property}, Congres SMAI 2013, 437--446, ESAIM Proc. Surveys 45, EDP Sci., Les Ulis, 2014.

\bibitem{D}
J. Dieudonn\'e, {\it Foundations of Modern Analysis},
Pure and Applied Mathematics~10-I, Academic Press, New York and London, 1969.

\bibitem{Du}
J. Duda, {\it Curves with finite turn},
Czechoslovak Math.\ J. {\bf 58}(133) (2008), 23–-49.

\bibitem{DZ}
J.~Duda and L.~Zaj\'\i\v{c}ek,   {\it Semiconvex functions: representations as suprema of smooth functions and extensions}, J. Convex Analysis {\bf 16} (2009),  239--260. 

\bibitem{DZ2}
J.~Duda and L.~Zaj\'\i\v{c}ek,  {\it On vector-valued curves that allow a $C^{1,\alpha}$-parametrization}, Acta Math.\ Hung. {\bf 127} (2010), 85-–111.  

\bibitem{Eg58}
H.G. Eggleston, 
{\it Notes on Minkowski geometry (I): Relations between the circumradius, diameter, inradius and minimal width of a convex set},
J. London Math.\ Soc. {\bf 33} (1958), 76--81.

\bibitem{Fe59} 
H. Federer, {\it Curvature measures}, Trans.\ Amer.\ Math.\ Soc. {\bf 93} (1959), 418--491.

\bibitem{Fu85} 
J.H.G. Fu, {\em Tubular neighborhoods in Euclidean spaces}, Duke Math.\ J. {\bf 52} (1985), 1025--1046.


\bibitem{HJ}
P. H\' ajek and M. Johanis, {\it Smooth Analysis in Banach Spaces},
De Gruyter Series in Nonlinear Analysis and Applications 19, De Gruyter, Berlin, 2014.

\bibitem{Hat01}
A. Hatcher, {\it Algebraic Topology}, Cambridge Univ.\ Press, Cambridge, 2002.

\bibitem{Hu}
D. Hug, {\it Generalized curvature measures and singularities of sets with positive reach},
Forum Math. {\bf 10} (1998), 699--728. 

\bibitem{HRW}
D. Hug, J. Rataj and W. Weil, {\it A product integral representation of mixed volumes of two convex bodies}, Adv.\ Geom. {\bf 13} (2013),  633--662. 	

\bibitem{Jo}
M. Johanis, {\it A quantitative version of the converse Taylor theorem: $C^{k,\omega}$-smoothness}, Colloq.\ Math. {\bf 136} (2014),  57--64.

\bibitem{Ly05} 
A. Lytchak, {\it Almost convex subsets}, Geom.\ Dedicata {\bf 115} (2005), 201--218.

\bibitem{Mil97} 
J.W. Milnor, {\it Topology from the Differentiable Viewpoint}, Princeton, New Jersey, 1997.
 
\bibitem{PR13} D. Pokorn\'y and J. Rataj, {\it Normal cycles and curvature measures of sets with d.c. boundary},  Adv.\ Math. {\bf 248} (2013), 963--985.

\bibitem{Rat02} 
J. Rataj, {\it Determination of spherical area measures by means of dilation volumes}, Math.\ Nachr. {\bf 235} (2002), 143--162.

\bibitem{Re}
Yu.G. Reshetnyak, {\it On a generalization of convex surfaces} (Russian),
Mat.\ Sb. N.S. {\bf 40}(82) (1956), 381--398. 

\bibitem{Sch}
R. Schneider, {\it Boundary structure and curvature of convex bodies}, In: Contributions to Geometry, J. T\"olke and J.M. Wills Eds., Birkhäuser, Basel-Boston, 1979, pp. 13--59.

\bibitem{Scho}
S. Scholtes, {\it On hypersurfaces of positive reach, alternating Steiner formulæ and Hadwiger’s Problem}, preprint, 2013, arxiv:1304.4179.


\bibitem{VZ} 
L. Vesel\'y and L. Zaj\'\i\v cek, {\it Delta-convex mappings between Banach spaces
and applications}, Dissertationes Math. (Rozprawy Mat.) {\bf 289} (1989), 52~pp.

\bibitem{zajcon}
L.~Zaj\'\i\v{c}ek, 
{\em On the differentiation of convex functions in finite and infinite dimensional spaces},  
Czechoslovak Math.\ J. {\bf 29} (1979), 340--348.

\bibitem{zajploch}
L.~Zaj\'\i\v{c}ek, 
{\em On Lipschitz and D.C. surfaces of finite codimension in a Banach space},  
Czechoslovak Math.\ J. {\bf 58} (2008), 849--864.
\end{thebibliography}
\end{document}